\documentclass[a4paper,oneside,11pt]{article} % Indique que le document est de type standard ("article"), au format A4 ("a4"), en recto simple ("oneside"), avec des fontes de taille moyenne ("11pt").

\usepackage[bottom]{footmisc}

\usepackage{a4wide}

\usepackage{amssymb}
\usepackage{tikz}
\usepackage{amsmath}

\usepackage{comment}

\usepackage{nameref}

\usepackage{bbm}

\usepackage{enumitem, hyperref}
\makeatletter
\def\namedlabel#1#2{\begingroup
    #2%
    \def\@currentlabel{#2}%
    \phantomsection\label{#1}\endgroup
}
\makeatother

\usepackage[all,cmtip]{xy}
\usepackage{tikz-cd}
\usepackage{graphicx}

\usepackage[all]{xy}

\usepackage{mathtools}
\DeclarePairedDelimiter\ev{\langle}{\rangle}

\usepackage[utf8]{inputenc} % Encodage des caractères du fichier-source.
\usepackage[T1]{fontenc} % Encodage des caractères en sortie.
\usepackage{textcomp} % Jeu de symboles complémentaires.
\usepackage[english]{babel}
\usepackage[autolanguage]{numprint} %% Formatage des nombres.
\usepackage{hyperref} % Génère des liens hypertexte dans le fichier pdf.
\usepackage{graphicx} % Permet l'insertion d'images.
\usepackage{verbatim} % Définit des environnements de texte préformaté et de commentaires.

\usepackage{amsmath,amssymb,amsthm} % Toutes ces extensions proviennent de la classe AMS-LaTeX.
\usepackage{comment}

\usepackage{fancyhdr} % Extension pour créer les en-têtes personnalisés.
\renewcommand\[{\begin{equation}}\renewcommand\]{\end{equation}} % Avec cette ligne, vos formules seront systématiquement numérotées.
\renewcommand\epsilon\varepsilon % Graphie du symbole "epsilon".
\renewcommand\phi\varphi % Graphie du symbole "phi".

\newtheorem{innercustomthm}{Theorem}
\newenvironment{customthm}[1]
  {\renewcommand\theinnercustomthm{#1}\innercustomthm}
  {\endinnercustomthm}

\newcommand\NNN{\mathbb{N}} % Ensemble des entiers naturels.
\newcommand\NN{\mathcal{N}}
\newcommand\TT{\mathcal{T}}
\newcommand\ZZ{\mathbb{Z}} % Ensemble des entiers relatifs.
\newcommand\ab\allowbreak % Raccourci pour \allowbreak.

\newcommand\VV{\mathcal{V}}

\newcommand\Spec{\operatorname{Spec}}

\newcommand\HI{\mathbf{HI}}

\newcommand\colim{\operatorname{colim}}

\newcommand\Hom{\operatorname{Hom}}

\newcommand\Tr{\operatorname{Tr}}

\newcommand\coker{\operatorname{coker}}
\newcommand\res{\operatorname{res}}
\newcommand\cores{\operatorname{cores}}
\newcommand\Id{\operatorname{Id}}
\newcommand\MW{\mathfrak{M}^{\operatorname{MW}}}

\newcommand\DM{\operatorname{DM}}
\newcommand\CCC{\mathcal{C}}
\newcommand\DMt{\widetilde{\operatorname{DM}}}

\newcommand\Flag{\operatorname{Flag}}

\newcommand\HM{\bold{HM}}

\newcommand\pizero{\underline{\pi}_0}

\newcommand\Sm{\operatorname{Sm}_k}

\newcommand\eeta{\boldsymbol{\eta}}
\newcommand\Om{\operatorname{\Omega}}
\newcommand\Ab{\mathcal{A}b}

\newcommand\KW{\underline{\operatorname{K}}^{W}}
\newcommand\KO{\underline{\operatorname{KO}}}
\newcommand\KMW{\underline{\operatorname{K}}^{MW}}

\newcommand\EEE{\hat{\mathbb{E}}}

\newcommand\VVV{\mathbb{V}}

\newcommand\EE{\mathbb{E}}

\newcommand\MMM{\hat{\mathbb{M}}}
\newcommand\MM{\mathbb{M}}

\newcommand\AAA{\mathbb{A}}

\newcommand\Tho{\operatorname{Th}}
\newcommand\rk{\operatorname{rk}}

\providecommand{\keywords}[1]
{
  \small	
  \textbf{\textit{Keywords---}} #1
}

\providecommand{\Codes}[1]
{
  \small	
  \textbf{\textit{MSC---}} #1
}

\newcommand\kMW{\mathbf{K}^{\text{MW}}}

\newcommand\LL{\mathcal{L}}

\newcommand\un{\mathbbm{1}}

\newcommand\Gm{\mathbb{G}_m}
\newcommand\SH{\mathbf{SH}}
\newcommand\PP{\mathbb{P}}
\newcommand\codim{\operatorname{codim}}

\theoremstyle{definition} % Pour tout ce qui ressemble à des définitions : définitions, notations...
\newtheorem{Def}{Definition}[subsection] % Définitions. Avec l'option "[section]", les définitions (et tout le reste) seront numérotées en fonction de la section courante.
\theoremstyle{plain} % Pour tous ce qui ressemble à des théorèmes.
\newtheorem{Pro}[Def]{Proposition} % Propositions. Avec l'option "[Def]", les propositions (et tout le reste) seront numérotées collectivement avec les définitions.
\newtheorem{Lem}[Def]{Lemma} % Lemmes.
\newtheorem{The}[Def]{Theorem} % Théorèmes.
\newtheorem{Cor}[Def]{Corollary} % Corollaires.
\theoremstyle{remark} % Pour tout ce qui ressemble à des remarques.
\newtheorem{Exe}[Def]{Example} % Exemples.
\newtheorem{Rem}[Def]{Remark} % Remarques.
\newtheorem{Par}[Def]{} % Paragraphe.

\title{Morel homotopy modules and Milnor-Witt cycle modules} % Remplacez "Titre" par votre titre.
\author{\sc Niels FELD\footnote{Adress: Institut Fourier, 100 Rue des Mathématiques, Grenoble, France.}
\footnote{E-mail adress: <niels.feld@univ-grenoble-alpes.fr>.} 
\footnote{Webpage: https://nielsfeld.wixsite.com/website} } % Remplacez "Auteur" par le nom du ou des auteurs.

\date{2019} % Avec l'argument "\today", la date indiquée sera celle du jour de la compilation.

\begin{document} % Début du document proprement dit.

\maketitle % Crée le bandeau de titre.
%\begin{abstract} % Résumé. Supprimer ces lignes si le texte ne s'y prête pas.
% Après quelques rappels concernant la théorie des catégories et celle des topos, on définit et étudie certaines propriété du topos effectif introduit par Hyland dans [?].
%\end{abstract}

%\tableofcontents % Table des matières. À réserver aux documents vraiment longs (au moins 25 pages).

\begin{abstract} We study the cohomology theory and the canonical Milnor-Witt cycle module associated to a motivic spectrum. We prove that the heart of Morel-Voevodsky stable homotopy category over a perfect field (equipped with its homotopy t-structure) is equivalent to the category of Milnor-Witt cycle modules, thus generalising Déglise's thesis. As a corollary, we recover a theorem of Ananyevskiy and Neshitov and we prove that the Milnor-Witt K-theory groups are birational invariants.
\end{abstract}


\Codes{14C17, 14C35, 11E81}
\tableofcontents

\section{Introduction}

% \subsection*{Historical approach}

In the fundamental paper \cite{Rost96}, Rost introduced the notion of a cycle module. The idea was to find a good axiomatization of the main properties encountered in the study of Milnor K-theory, Quillen K-theory or Galois cohomology. According to \cite{Rost96}, a cycle module $M$ over a perfect field $k$ is the data of a $\ZZ$-graded abelian group $M(E)$ for every finitely generated field extension $E/k$, equipped with restriction maps, corestriction maps, a Milnor K-theory module action and residue maps $\partial$. Moreover, these data are subject to certain compatibility relations $(r1a),\dots, (r3e), (fd)$ and $(c)$. The theory results in the construction of Gersten type complexes whose cohomology groups are called {\em Chow groups with coefficients} and can be used, for instance, to extend to the left the localization sequence of Chow groups associated with a closed embedding.
\par In order to construct the derived category of motives $\DM(k,\ZZ)$, Voevodsky introduced the so-called homotopy sheaves (with transfers) which are homotopy invariant Nisnevich sheaves with transfers. One important example is given by $\Gm$, the sheaf of global units. Voevodsky proved that any homotopy sheaf $F$ has a Gersten resolution, implying that $F$ is determined in some sense by the data of its fibers in every function fields. This statement was made more precise in Déglise's thesis: the heart of $\DM(k,\ZZ)$ with respect to its homotopy t-structure has a presentation given by the category of Rost cycle modules over $k$.
\par Morel's point of view on the heart of $\DM(k,\ZZ)$ is given by the category of oriented homotopy modules. We recall that a homotopy module is a strictly $\AAA^1$-invariant Nisnevich sheaf with an additional structure defined over the category of smooth schemes (see Definition \ref{DefHomotopyModules}); it is called {\em oriented} when the Hopf map $\eeta$ acts on it by $0$. Déglise's theorem proves that oriented homotopy modules form a subcategory of the category of homotopy modules which is equivalent to the category of Rost cycle modules. Morel's natural conjecture \cite[Remark 2.49]{Mor12} was that there is a presentation of the heart of the stable homotopy category $\SH(k)$ (or equivalently, the category of homotopy modules) in terms of some non-oriented version of cycle modules.

\subsection{Current work}
In \cite{Fel18}, we introduced the theory of Milnor-Witt cycle modules, generalising the work of Rost \cite{Rost96} on cycle modules and Schmid's thesis 	\cite{Schmid98}.	Indeed, we have studied	 general cycle complexes $C^*(X,M,\VV_X)$ and their (co)homology groups $A^*(X,M,\VV_X)$ (called {\em Chow-Witt groups with coefficients}) in a quadratic	setting over a perfect base field of any characteristic. The general coefficient systems $M$ for these complexes are called Milnor-Witt cycle modules. The main example of such a cycle module is given by Milnor-Witt K-theory (see \cite[Theorem 4.13]{Fel18}); other examples will be deduced from Theorem \ref{ThmDeg} or Theorem \ref{AdjunctionTheoremBis} (e.g. the representability of hermitian K-theory in $\SH(k)$ will lead to a MW-cycle module, associated with hermitian  K-theory). A major difference with Rost's theory is that the grading to be considered is not $\ZZ$ but the category of virtual bundles (or, equivalently, the category of virtual vector spaces), where a virtual bundle $\VV$ is, roughly speaking, the data of an integer $n$ and a line bundle $\mathcal{L}$ (see \cite[Appendix A]{Fel18}). Intuitively, Milnor-Witt cycle modules are given by (twisted) graded abelian groups equipped with extra data (restriction, corestriction, $\KMW$-action and residue maps).
\par For any scheme $X$, any virtual bundle $\VV_X$ and any Milnor-Witt cycle module $M$, we have proved that there exists a complex $C^*(X,M,\VV_X)$ equipped with pushforwards, pullbacks, a Milnor-Witt K-theory action and residue maps satisfying standard functoriality properties. A fundamental theorem is that the associated cohomology groups $A^*(X,M,\VV_X)$ satisfy the homotopy invariance property (see \cite[Theorem 9.4]{Fel18}).
\par In this paper, we prove that Milnor-Witt cycle modules are closely related to Morel's $\AAA^1$-homotopy theory: they can be realized geometrically as elements of the stable homotopy category. Precisely, we prove the following theorem.
\begin{customthm}{1}[Theorem \ref{ThmDeg}]
Let $k$ be a perfect field. The category of Milnor-Witt cycle modules is equivalent to the heart of Morel-Voevodsky stable homotopy category (equipped with the homotopy t-structure):
\begin{center}

$ \mathfrak{M}^{MW}_k \simeq {\SH(k)}^\heartsuit$.
\end{center}

\end{customthm}
In order to prove this theorem, we study the cohomology theory associated with a motivic spectrum. This notion is naturally dual to the bivariant theory developed in \cite{DJK18} and recalled in Section \ref{BivTheory} (see Theorem \ref{Duality}). A motivic spectrum $\EE$ leads to a functor $\EEE$ from the category of finitely generated fields over $k$ to the category of graded abelian groups (be careful that the grading is not $\ZZ$ but is given by the category of virtual vector spaces). We prove that the functor $\EEE$ is a Milnor-Witt cycle premodule (see \cite[Definition 3.1]{Fel18}). Indeed, most axioms are immediate consequences of the general theory \cite{DJK18}. Moreover, in Theorem \ref{RamificationGeom} we prove a ramification theorem of independent interest that can be applied to prove rule \ref{itm:R3a}. Furthermore, we check axioms \ref{itm:FD} and \ref{itm:C} so that $\EEE$ is a Milnor-Witt cycle module. These two axioms follow from the study of a spectral sequence defined in Section \ref{HMandMW} (see Theorem \ref{DifferentialsComputation}); another -- more elementary -- proof may result from an adaptation of \cite[Theorem 2.3]{Rost96} to the context of Milnor-Witt cycle modules but this method would rely heavily on the fact that the base field is perfect.
\par In Section \ref{EquivalenceOfCat}, we construct a homotopy module for any Milnor-Witt cycle module and proceed to prove that the heart of the stable homotopy category (which is known to be equivalent to the category of homotopy module) is equivalent to the category of Milnor-Witt cycle modules.
\par This result generalizes Déglise's thesis (see Theorem \ref{theseDeglise}) and answers affirmatively an old conjecture of Morel (see \cite[Remark 2.49]{Mor12}). An important corollary is the following result (which was proved independently by Ananyevskiy and Neshitov in \cite[Theorem 8.12]{Neshitov2018}):
\begin{customthm}{2}[Theorem \ref{ThmAnaNeshi}]
The heart of Morel-Voevodsky stable homotopy category is equivalent to the heart of the category of MW-motives \cite{DegFas18} (both equipped with their respective homotopy t-structures):
\begin{center}
${\SH(k)}^\heartsuit \simeq {\DMt(k)}^\heartsuit$.

\end{center}
\end{customthm}

\begin{comment}

We conclude this paper with the study of a conjecture of  Bachmann and Yakerson (see \cite{BachmannYakerson18}):
\begin{The}

Let $n\geq 1$. Then the canonical functor
\begin{center}

${\Sigma^{\infty-n}_{\Gm}}^{\heartsuit}:{\operatorname{SH}^{S^1}(k)(n)}^{\heartsuit}\to {\operatorname{SH}^{eff}(k)}^{\heartsuit}$
\end{center}
is an equivalence of abelian categories.
\end{The}

\end{comment}
%\par In future work, we will study the theory of Milnor-Witt cycle modules over more general base scheme $S$ and prove Theorem \ref{ThmDeg} in this more general setting. Assuming such task accomplished, we may already claim the corollaries stated in Section \ref{Applications} to be true (including Theorem \ref{ThmAnaNeshi} whose original proof \cite{Neshitov2018} needs the base field to be perfect).

Finally, we give an application of our theory to birational questions:
\begin{The}[Theorem \ref{BirInv}]
	Let $X$ be a proper smooth integral scheme over $k$, let $\VV_k$ a virtual vector bundle over $k$ and let $M$ be a Milnor-Witt cycle module. Then the group $A^0(X,M,-\Om_{X/k}+\VV_X)$ is a birational invariant of $X$ in the sense that, if $X\dashrightarrow Y $ is a birational map, then there is an isomorphism of abelian groups
\begin{center}
$A^0(Y,M,-\Om_{Y/k}+\VV_Y) \to A^0(X,M,-\Om_{X/k}+\VV_X)$.
\end{center}
In particular for $M=\KMW$, we obtain the  fact that the Milnor-Witt K-theory groups $\kMW_n$ are birational invariants.
\par Moreover, if $F\in \HI(k)$ be a homotopy sheaf, then $F(X)$ is a birational invariant.
\end{The}
We hope that the `quadratic' nature of these new invariants could lead to more refined theorems in the domain.
\subsection{Outline of the paper}

In Section \ref{BivTheory}, we follow \cite{DJK18} and define the bivariant theory associated to a motivic spectrum. We extend the main results for the associated cohomology theory. We study the basic properties of fundamental classes and prove a ramification formula.
\par In Section \ref{HMandMW}, we recall the theory of Milnor-Witt cycle modules developed in \cite{Fel18}. For any motivic spectrum, we then construct a Milnor-Witt cycle modules in a functorial way. 
\par The heart of the paper is Section \ref{EquivalenceOfCat} where we define a homotopy module for any Milnor-Witt cycle module and prove our main theorem: the heart of the stable homotopy category (which is known to be equivalent to the category of homotopy module according to \cite{Mor03}) is equivalent to the category of Milnor-Witt cycle modules.
\par Finally in Section \ref{Applications}, we give some corollaries of the main theorem. In particular, we show that the heart of stable homotopy category is equivalent to the heart of the category of MW-motives \cite{DegFas18}. We also prove that Milnor-Witt K-theory groups and homotopy sheaves are birational invariants of smooth projective $k$-schemes.

\subsection{Notation}\label{Conventions}
Throughout the paper, we fix a (commutative) field $k$ and we assume moreover that $k$ is perfect (of arbitrary characteristic). We consider only schemes that are noetherian and essentially of finite type\footnote{That is, isomorphic to a limit of finite type schemes with affine \'etale transition maps.} over $k$. All schemes and morphisms of schemes are defined over $k$.
\par We denote by $S=\Spec k$ the spectrum of $k$.
%\par We use the term {\em s-morphism} as an abbreviation for {\em separated morphism of finite type}.
\par By a field $E$ over $k$, we mean {\em a $k$-finitely generated field $E$}. Since $k$ is perfect, notice that $\Spec E$ is essentially smooth over $S$.
% \par Denote by $\mathbb{F}_k$ the category of fields over $k$ (with obvious morphisms). 
% \par If $A$ is a commutative ring, denote by $\mathbb{V}(A)$ the category of projective $A$-modules of finite type and $\mathfrak{V}(A)$ the category of virtual projective $A$-modules of finite type (see \cite{Deligne87} which uses the notation $\underline{K}(A)$; see also Appendix \ref{VirtualObj} for more details). Recall that there is a {\em contravariant} equivalence functor from $\mathbb{V}(A)$ to the category of vector bundles over $X=\Spec A$ (see \cite{EGA1}). We will sometimes go from one category to the other without mentionning this functor.
\par Let $f:X\to S$ be a morphism of schemes and $\VV_S$ be a virtual bundle over $S$. We denote by $\VV_X$ or by $f^*\VV_S$  or by $\VV_S\times_S X$ the pullback of $\VV_S$ by $f$.
\par Let $f:X\to Y$ be a morphism of schemes. Denote by $\LL_f$ or by $\LL_{X/Y}$ the virtual vector bundle over $Y$ associated with the cotangent complex of $f$. If $p:X\to Y$ is a smooth morphism, then $\LL_p$ is (isomorphic to) $\TT_p=\Om_{X/Y}$ the space of (Kähler) differentials. If $i:Z\to X$ is a regular closed immersion, then $\LL_i$ is the normal cone $-\NN_ZX$. If $f$ is the composite $\xymatrix{ Y \ar[r]^i & \PP^n_X \ar[r]^p & X}$  with $p$ and $i$ as previously (in other words, if $f$ is lci projective), then $\LL_f$ is isomorphic to the virtual tangent bundle $i^*\TT_{\PP^n_X/X} - \NN_Y(\PP^n_X) $ (see also \cite[Section 9]{Fel18}).
% \par Let $F/E$ be a field extension, we denote by $\Om_{F/E}$ the $F$-vector space of (Kähler) differentials. We use the same notation to denote its canonical representant in the category of virtual vector spaces. Dually, let $X/S$ be a scheme morphism, we denote by $\TT_{X/S}$ the sheaf of modules of differentials (or rather, its virtual representant).
\par Let $E$ be a field (over $k$) and $v$ a (discrete) valuation on $E$. We denote by $\mathcal{O}_v$ its valuation ring, by $\mathfrak{m}_v$ its maximal ideal and by $\kappa(v)$ its residue class field. We consider only valuations of geometric type, that is we assume: $k\subset \mathcal{O}_v$, the residue field $\kappa(v)$ is finitely generated over $k$ and satisfies $\operatorname{tr.deg}_k(\kappa(v))+1=\operatorname{tr.deg}_k(E)$.
% \par Let $E$ be a field and $v$ be a valuation on $E$. We denote by $\NN_v$ the $\kappa(v)$-vector space $\mathfrak{m}_v/\mathfrak{m}_v^2$ and call it the {\em normal cone of $v$}.
\par For $E$ a field (resp. $X$ a scheme), we denote by $\ev{n}$ the virtual space $\AAA^n_E$ (resp. $\AAA^n_X$).

\section{Bivariant theory} \label{BivTheory}
\subsection{Recollection and notations}

In this subsection, we recall some results from \cite[§2]{DJK18}. Let $S$ be a base scheme. Denote by $\VVV(S)$ the Picard groupoid of virtual vector bundles on $S$ (see \cite[§4]{Deligne87} or \cite[Appendix A]{Fel18}). If $\VV_S$ is a virtual vector bundle over $S$, we denote by $\Tho_S(\VV_S)$ its associated Thom space (this is an $\otimes$-invertible motivic spectra over $S$, see \cite[Remark 2.4.15]{CD12}).

\begin{Def} \label{DefBivTheory}

Let $\EE\in \SH(S)$ be a motivic spectrum. Given a separated morphism of finite type $p:X\to S$, an integer $n\in \ZZ$ and a virtual bundle $\VV_X\in \VVV(X)$, we define the {\em bivariant theory} of $X/S$ in bidegree $(n,\VV_X)$, with coefficients in $\EE$, as the abelian group:
\begin{center}
$\EE_n(X/S,\VV_X)=[\Tho_X(\VV_X)[n],p^!(\EE)]=[p_!\Tho_X(\VV_X)[n],\EE]$.
\end{center}
The {\em cohomology theory} represented by $\EE$ is defined by the formula:
\begin{center}

$\EE^n(X,\VV_X)=\EE_{-n}(X/X,-\VV_X)=[\un_X, \EE_X\otimes \Tho_X(\VV_X)[n]]$
\end{center}
for any scheme $X$ over $S$ and any pair $(n,\VV_X)\in \ZZ\times \VVV(X)$.
\par In the special case where $\EE=\un$ is the sphere spectrum, we will use the notation
\begin{center}

$H_n(X/S,\VV_X)=\un_n(X/S,\VV_X)=[\Tho_X(\VV_X)[n],p^!(\un_S)]$
\end{center}
and we will refer to this simply as the {\em bivariant $\AAA^1$-theory}.
\par Similarly, we set $H^n(X,\VV_X)=\un^n(X,\VV_X)$ and refer to this as the {\em $\AAA^1$-cohomology}.

\end{Def}

\begin{Par} \label{BaseChangeBiv}
{\sc base change.} For any cartesian square
  \begin{center}

  $\xymatrix{
  Y \ar[r]^g \ar[d]_q \ar@{}[rd]|-\Delta  & X \ar[d]^p \\
  T \ar[r]_f & S,
  }$
  \end{center}
 one gets a map
 \begin{center}
 
 $\Delta^*:\EE_n(X/S,\VV_X)\to \EE_n(Y/T, \VV_T)$
 \end{center}
 by applying the functor $g^*:\SH(X)\to \SH(Y)$ and using the exchange transformation ${\operatorname{Ex}^{*!}:g^*p^! \to q^! f^*}$ associated with the square $\Delta$.
\end{Par}

\begin{Par}{\sc Covariance for proper morphisms.} \label{covarianceBiv}
Let $f:Y\to X$ be a proper morphism. We have a map
\begin{center}

$f_*:\EE_n(Y,\VV_Y)\to \EE_n(X,\VV_X)$

\end{center}
coming from the unit map $f_!f^!\to \Id$ and the fact that $f_!=f_*$ since $f$ is proper.
\end{Par}
\begin{Par}
{\sc Contravariance for étale morphisms.} Let $f:Y\to  X$ be an étale morphism, we have a map
\begin{center}

$f^*:\EE_n(X/S,\VV_X)\to \EE_n(Y/S,\VV_Y)$
\end{center}
obtained by applying the functor $f^!:\SH(X)\to \SH(S)$ and using the purity isomorphism $f^!=f^*$ as $f$ is étale.
\end{Par}

\begin{Par}{\sc Products.} \label{BivProduct}
Consider a multiplication map $\mu:\EE \otimes \EE' \to \EE''$ between motivic spectra. For any s-schemes $\xymatrix{ Y \ar[r]^q & X \ar[r]^p & S}$, any integers $n,m$ and any virtual vector bundles $\mathcal{W}_Y/Y$ and $\VV_X/X$, there is a multiplication map
\begin{center}

$\EE_m(Y/X,\mathcal{W}_Y)\otimes \EE'_n(X/S,\VV_X) \to \EE''_{m+n}(Y/S,\mathcal{W}_Y+\VV_Y)$.
\end{center}
\end{Par}
 
\begin{Def}

Let $X \to S$ be a separated morphism of finite type.
\begin{itemize}
\item A {\em fundamental class}\footnote{Also called {\em orientation} in \cite[Definition 2.3.2]{DJK18}.} of $f$ is an element
\begin{center}

$\eta_f \in H_0(X/S,\VV_f)$
\end{center}
for a given virtual vector bundle $\VV_f$ over $X$.
\item Let $\mathcal{C}$ be a subcategory of the category of (quasi-compact and quasi-separated) schemes. A {\em system of fundamental classes} for $\mathcal{C}$ is the data, for each morphism $f:X\to Y$ in $\mathcal{C}$, of a virtual bundle $\VV_f\in \mathbb{V}(X)$ and an orientation $\eta^{\mathcal{C}}_f\in H_0(f,\VV_f)$ such that the following relations hold:
\begin{enumerate}
\item {\em Normalisation.} If $f=\Id_S$, then $\VV_f=0$ and the orientation $\eta^{\mathcal{C}}_f\in H_0(\Id_S,0)$ is given by the identity $\Id:\un_S\to \un_S$.
\item {\em Associativity formula.} For any composable morphisms $f$ and $g$ in $\mathcal{C}$, one has an isomorphism:
\begin{center}

$\VV_{f\circ g}\simeq \VV_g+g^*\VV_f$
\end{center}
and, modulo this identification, the following relation holds:
\begin{center}

$\eta^\mathcal{C}_g.\eta^\mathcal{C}_f=\eta^\mathcal{C}_{f\circ g}$
\end{center}
\end{enumerate}
\item Suppose the category $\mathcal{C}$ admits fibred products. We say that a system of fundamental classes $(\eta^\mathcal{C}_f)_f$ is {\em stable under transverse base change} if it satisfies the following condition: for any cartesian square
\begin{center}

$\xymatrix{ 
Y \ar[r]^g \ar[d]_q \ar@{}[rd]|-{\Delta} & T \ar[d]^p \\
X \ar[r]_f & S.}$
\end{center}
such that $f$ and $g$ are in $\mathcal{C}$ and $p$ is transverse to $f$, then one has $\VV_g=q^*\VV_f$ and the following formula holds in $H_0(g,\VV_g)$: $\Delta^*(\eta^\mathcal{C}_f)=\eta^\mathcal{C}_g$.
\end{itemize}
\end{Def}
In \cite[Thoerem 3.3.2]{DJK18}, the authors prove the following theorem.
\begin{The}

There exists a unique system of fundamental classes $\eta_f\in H_0(X/S,\LL_f)$ associated with the class of quasi-projective lci morphisms $f$ such that:
\begin{enumerate}
\item For any smooth separated morphism of finite type $p$, the class $\eta_p$ agrees with the fundamental class defined in \cite[Example 2.3.9]{DJK18} thanks to the purity isomorphism.
\item For any regular closed immersion $i:Z\to X$, the class $\eta_i$ agrees with the fundamental class defined in \cite[Theorem 3.2.21]{DJK18} by deformation to the normal cone.
\end{enumerate}

\end{The} 

Thanks to this system of fundamental classes, we can define Gysin morphisms as follows.
\begin{Par}{\sc Contravariance for lci morphisms.} \label{BivLci}
Let $f:Y\to X$ be an lci quasi-projective morphism $f:Y/S \to X/S$ of separated morphisms of finite type with virtual cotangent bundle $\LL_f$. For any motivic spectrum $\EE\in \SH(S)$, for any integer $n$ and any virtual bundle $\VV_X$ over $X$, there exists a Gysin morphism:
\begin{center}

$f^*:\EE_n(X/S,\VV_X)\to \EE_n(Y/S,\LL_f+\VV_Y)$ \\
$x\mapsto \eta_f.x$
\end{center}
using the product defined in \ref{BivProduct}.
These Gysin morphisms satisfy the following formulas:
\begin{enumerate}
\item {\em Functoriality:} \label{lciFunctoriality} For any suitable morphisms $f$ and $g$, one has $(fg)^*=g^*f^*$.
\item {\em Base change:} For any cartesian square
\begin{center}

$\xymatrix{ 
Y \ar[r]^g \ar[d]_q \ar@{}[rd]|-{\Delta} & T \ar[d]^p \\
X \ar[r]_f & S.}$

\end{center}
such that $f$ is quasi-projective lci and transverse to $p$, one has $f^*p_*=q_*g^*$. 
\end{enumerate}
\end{Par}

%\begin{Par} {\sc Milnor-Witt action} \label{MWactionGeom}
%Any motivic spectrum $\EE$ in the heart is equipped with a unit isomorphism $\un \otimes \EE \to \EE$ which defines an action by composition on the left
%\begin{center}
%
%$\gamma:\un^m(X,\mathcal{W}_X)\otimes \EE^n(X,\VV_X)\to \EE^{m+n}(X,\mathcal{W}_X+\VV_X)$.
%\end{center}
%\end{Par}

\begin{Par} {\sc Localization long exact sequence.} \label{LocalizationMap}
Another essential property is the following long exact sequence following from the usual localization triangle in the six functors formalism in $\SH$.
\par Indeed, let $\EE\in \SH(S)$ be a motivic spectrum. For any closed immersion $i:Z\to X$ of separated schemes over $S$, with (quasi-compact) complementary open immersion $j:U\to X$, there exists a canonical localization long exact sequence of the form:
\begin{center}

$\xymatrixcolsep{1pc}\xymatrix{
\dots \ar[r] &
\EE_n(Z/S,\VV_Z) \ar[r]^{i_*} &
\EE_n(X/S,\VV_X) \ar[r]^{j^*} &
\EE_{n}(U/S,\VV_U) \ar[r]^-{\partial_i} &
\EE_{n-1}(Z/S,\VV_Z) \ar[r] &
\dots.
}$
\end{center}
\end{Par}

We will need the following properties of localizations long exact sequences.

\begin{Pro} \label{LocalizationBiv}

Let $\EE\in \SH(S)$ be a motivic spectrum and consider the following commutative square
\begin{center}

$\xymatrix{
T \ar@{^{(}->}[r]^k  \ar@{^{(}->}[d]_q  & Y \ar@{^{(}->}[d]^p \\
Z \ar@{^{(}->}[r]_i & X 
}$

\end{center}
of closed immersions of separated schemes over $S$. For $\VV_X$ a virtual vector bundle over $X$, we have the following diagram
\begin{center}

$\xymatrixcolsep{1pc} \xymatrix{
\EE_n(T/S,\VV_T) \ar[r]^{k_*} \ar[d]_{q_*} &
\EE_n(Y/S,\VV_Y) \ar[r]^{k'^*} \ar[r] \ar[d]^{p_*} &
\EE_n(Y- T, \VV_{Y- T}) \ar[r]^{\partial_k} \ar[d]^{\tilde{p}_*} &
\EE_{n-1}(T/S,\VV_T) \ar[d]^{q_*} \\
\EE_n(Z/S,\VV_Z) \ar[r]^{i_*} \ar[d]^{q'^*} & 
\EE_n(X/S,\VV_X) \ar[r]^{i'^*} \ar[d]^{p'_*} &
\EE_n(X- Z/S,\VV_{X- Z}) \ar[d]^{\tilde{p}'^*} \ar[r]^{\partial_i} &
\EE_{n+1}(Z/S,\VV_Z) \ar[d]^{q'^*} \\
\EE_n(Z- T/S,\VV_{Z- T}) \ar[r]^{\tilde{i}_*} \ar[d]^{\partial_k} &
\EE_n(X- Y/S,\VV_{X- Y}) \ar[r]^-{\tilde{i}'^*} \ar[d]^{\partial_{{p}}} &
\EE_n(X- (Z\cup Y)/S, \VV_{X- (Z\cup Y)}) \ar@{}[rd]|-{(*)} \ar[d]^{\partial_{\tilde{p}}} \ar[r]^-{\partial_{\tilde{i}}} &
\EE_{n-1}(Z- T/S,\VV_{Z- T}) \ar[d]^{\partial_{\tilde{k}}} \\
\EE_{n-1}(T/S,\VV_T) \ar[r]_{k_*} & 
\EE_{n-1}(Y/S,\VV_Y) \ar[r]_{k'^*} &
\EE_n(Y- T/S,\VV_{Y- T}) \ar[r]_{\partial_k} &
\EE_{n-2}(T/S,\VV_T)
}$
\end{center}
with obvious maps. Each squares of this diagram is commutative except for $(*)$ which is anti-commutative.
\end{Pro}
\begin{proof}See \cite[Proposition 2.2.11]{DJK18}.
\end{proof}

\begin{Pro}[Base change for lci morphisms]\label{BaseChangelci} 
Consider a cartesian square of schemes
\begin{center}

$\xymatrix{
X' \ar[r]^{f'} \ar[d]_{g'} & Y' \ar[d]^g \\
X \ar[r]_f & Y
}$
\end{center}
with $f$ proper, and $g$ (quasi-projective) lci.
Suppose moreover that the square is {\em tor-independent}, that is for any $x\in X$, $y'\in Y'$ with $y=f(x)=g(y')$ and for any $i>0$ we have
\begin{center}

$\operatorname{Tor}^{\mathcal{O}_{Y,y}}_i(\mathcal{O}_{X,x},\mathcal{O}_{Y',y'})=0$.
\end{center}
Up to the canonical isomorphism $f^{\prime *}\LL_g \simeq \LL_{g'}$,we have:
\begin{center}
$f'_* \circ g'^* = g^*\circ f_*$.
\end{center}
\end{Pro}
\begin{proof} See \cite[Proposition 4.1.3]{DJK18}.
\end{proof}

\begin{comment}

\end{comment}

The bivariant theory of a spectrum satisfies some $\AAA^1$-homotopy invariance property.
\begin{Pro} \label{BivA1invariance}
Let $\EE\in \SH(S)$ be a motivic spectrum. Let $X$ be an s-scheme over $S$ and let $p:V\to X$ be a vector bundle with virtual tangent bundle $\LL_p=p^*\langle V \rangle$. Then the Gysin morphism
\begin{center}

$p^*:\EE_n(X/S,\VV_X)\to \EE_n(V/S,\LL_p+\VV_V)$
\end{center}
is an isomorphism for any integer $n$ and any virtual bundle $\VV_X$ over $X$.
\end{Pro}

\begin{proof}
See \cite[Lemma 2.4.4]{DJK18} when $\EE$ is the sphere spectrum $\un$. The general case follows from the definitions.
\end{proof}
\begin{Def}
Keeping the notations of \ref{BivA1invariance}, we define the {\em Thom isomorphism}
\begin{center}

$\Phi_{V/X}: \EE_n(V/S,\VV_V)\to \EE_n(X/S,\VV-\langle V \rangle)$,
\end{center}
associated with $V/X$, as the inverse of the Gysin morphism $p^*:\EE_n(X/S,\VV_X-\langle E \rangle)\to \EE_n(V/S,\VV_V)$
\end{Def}
\begin{Rem} \label{BaseChangeThomIso}
The Thom isomorphism is compatible with base change and direct sums (see \cite[Remark 2.4.6]{DJK18}).
\end{Rem}

\subsection{Ramification formula}

\begin{Par} \label{RamificationFormula}

\par 
Consider a topologically cartesian square
\begin{center}

$\xymatrix{
T \ar@{^{(}->}[r]^k  \ar[d]_q \ar@{}[rd]|-{\Delta} & Y \ar[d]^p \\
Z \ar@{^{(}->}[r]_i & X 
}$
\end{center}
where $k$ and $i$ are regular closed immersions of codimension 1 and $T$ is a reduced connected scheme. We have $T=(Z\times_XY)_{\operatorname{red}}$. Denote by $\mathfrak{I}'$ the ideal of the immersion $Z\times_X Y \to Y$, by $\mathfrak{I}$ the ideal of $k$ and by $\mathfrak{J}$ the ideal of $i$. Assume moreover that the square is ramified with ramification index $e$ in the sense that there exists a nonzero natural number $e$ such that $\mathfrak{I}'=\mathfrak{I}^e$. We consider the morphism of deformation spaces ${\nu^{(e)}: D_TY\to D_ZX}$ defined as the spectrum of the composite
\begin{center}

$\xymatrix{
\bigoplus_{n\in \ZZ} \mathfrak{J}^n\cdot t^{-n} \ar[r] & 
\bigoplus_{n\in \ZZ} (\mathfrak{I}^e)^n\cdot t^{-n} \ar[r] & 
\bigoplus_{m\in \ZZ} \mathfrak{I}^m\cdot t^{-m} \\
x\cdot t^{-n} \ar@{|->}[r] & \tilde{f}(x)\cdot t^{-n} \ar@{|->}[r] & \tilde{f}(x)\cdot t^{-en} 
}$

\end{center}
where the first map is induced by the morphism $\tilde{f}:\mathfrak{J}\to \mathfrak{I}'$ defined via $f$ and where the second map takes the parameter $t$ to its power $t^e$. 
\par The map $\nu^{(e)}$ factors naturally making the following diagram commutative:
\begin{center}

$\xymatrix{
N_TY \ar@{^{(}->}[r] \ar[d] & D_TY \ar[d] & \Gm Y \ar@{=}[d] \ar@{_{(}->}[l] \\
q^*N_ZX \ar@{^{(}->}[r] \ar[d] & q^*D_ZX \ar[d] & \Gm Y \ar[d] \ar@{_{(}->}[l] \\
N_ZX \ar@{^{(}->}[r] & D_ZX & \Gm X. \ar@{_{(}->}[l]
}$
\end{center}	
Hence, we have the following commutative diagram:
\begin{center}

$\xymatrixcolsep{5pc}
\xymatrix{
\EE_{n+1}(\Gm Y/Y,*) \ar@{=}[d] \ar@{}[rd]|-{(1)} \ar[r]^-{\partial_{T/Y}} &
 \EE_n(N_TY/Y,*) \ar[d] \ar@{}[dr]|-{(3)}\ar[r]^-{\Phi_{N_TY/T}}_-{\sim} &
 \EE_n(T/Y,-\NN_TY+*) \ar[d]^-{\nu^{(e)}_*} \\
 \EE_{n+1}(\Gm Y/Y,*) \ar[r] \ar@{}[rd]|-{(2)}  &
 \EE_{n}(q^* N_ZX/Y, *) \ar[r]^-{\Phi_{q^* N_ZX/T}}_-{\sim} \ar@{}[rd]|-{(4)} &
 \EE_n(T/Y,-q^*\NN_ZX+*) \\
 \EE_{n+1}(\Gm X/X,*) \ar[u]^{\Delta^*} \ar[r]_{\partial_{Z/X}} &
 \EE_n(N_ZX/X,*) \ar[r]_{\Phi_{NZX/Z}}^-{\sim} \ar[u]^{\Delta^*} &
 \EE_n(Z/X,-\NN_ZX+*) \ar[u]^{\Delta^*}
 }$
\end{center}
where the arrows $\Delta^*$ denote the obvious maps induced by the corresponding squares. Square (1) (resp. square (2)) is commutative because of the naturality of localization long exact sequences with respect to the proper covariance (resp. base change). The map $\nu^{(e)}_*$ is defined so that the square (3) commutes. Square (4) commutes by compatibility of Thom isomorphisms with respect to base change (\ref{BaseChangeThomIso}).
\par From this, we deduce the formula 
\begin{center}

$\Delta^*(\eta_i)=\nu_*^{(e)}(\eta_k)$.
\end{center}
 We have proved the following theorem:

\end{Par}

% A FAIRE : présenter mieux ce théorème
\begin{The} \label{RamificationGeom}

Keeping the previous notations,
the following holds in $\EE_n(T/Y, N_TY)$
\begin{center}

$\Delta^*(\eta_i)=\nu_*^{(e)}(\eta_k)$.
\end{center}
\end{The}

\begin{Cor}
Consider the topologically cartesian square
\begin{center}
	
	$\xymatrix{
		T \ar@{^{(}->}[r]^k  \ar[d]_q \ar@{}[rd]|-{\Delta} & Y \ar[d]^p \\
		Z \ar@{^{(}->}[r]_i & X 
	}$
\end{center}
where $k$ and $i$ are regular closed immersions of codimension 1 and $T$ is a reduced connected scheme. Assume moreover that the square is ramified with ramification index $e$ as before. Then we have the following ramification formula:
\begin{center}
$p^*i_*=\nu_*^{(e)} k_*q^*$
\end{center}
\end{Cor}

\begin{Rem}
One could say that $\nu_*^{(e)}$ represents the defect of transversality.
\end{Rem}

\subsection{Application to cohomology}

For a spectrum $\EE\in \SH(S)$, a natural number $n\in \NNN$, a morphism $p:X\to S$ of schemes and $\VV_X$ a virtual vector bundle over $X$, we define the cohomology group
\begin{center}
$\EE^n(X,\VV_X)=\Hom_{\SH(X)}(\un_X,\EE_X\otimes \Tho_X(\VV_X)[n])$

\end{center}
where $\EE_X=p^*\EE_S$.
\par It is dual to the bivariant theory $\EE_*(-,*)$ defined previously.
Indeed, we have the following theorem:

\begin{The} \label{Duality}
Let $f:X\to S$ be an essentially\footnote{In \cite{DJK18}, the authors worked only with separated $S$-schemes of finite type but we can extend in a canonical way most of the results for separated $S$-schemes \textit{essentially} of finite type (see also \cite[§2.1.1]{ADN19}).} smooth scheme and $\EE\in \SH(S)$ a motivic spectrum. Then for any integer $n$ and any virtual bundle $\VV_X$ over $X$, there is a canonical isomorphism
\begin{center}

$\EE^n(X,\VV_X)\simeq \EE_{-n}(X/S,\LL_f-\VV_X)$
\end{center}
which is contravariantly natural in $X$ with respect to étale morphisms.
\begin{Exe} \label{MainExampleCoh} A crucial example follows from the work of Morel: if $\EE$ is the unit sphere $\un$ and $X=\Spec E$ is the spectrum of a field, then the group $H^n(X,\ev{n})$ is isomorphic to the Milnor-Witt theory $\kMW_n(E)$.

\end{Exe}

\end{The}
In the following, we give the usual properties of the bivariant cohomology theory. The proofs follow directly by duality and the preceding results. In practice (since our base field $k$ is perfect) we work mainly with essentially smooth schemes, hence we could also apply Theorem \ref{Duality}.

%  \begin{itemize}
%  \item \underline{base change} for a cartesian square
%  \begin{center}
%  
%  $\xymatrix{
%  Y \ar[r]^g \ar[d]_q \ar@{}[rd]|-\Delta  & X \ar[d]^p \\
%  T \ar[r]_f & S,
%  }$
%  \end{center}
%  we have a map
%  \begin{center}
%  
%$\Delta^*:\EE^n(X/S,\VV^\LL_X) \to \EE^n(Y/T,\VV^\LL_Y). $
%\end{center}%
%
%\item \underline{covariance for a proper $S$-morphism} $f:Y\to X$: we have
%\begin{center}
%
%$f_*:\EE^n(Y/S,\VV^\LL_Y)\to \EE^n(X/S,\VV^\LL_X)$.
%\end{center}
%\item \underline{contravariance for essentially smooth $S$-morphism} $f:Y\to X$: we have
%\begin{center}
%
%$f^*:\EE^n(X/S,\VV_X)\to \EE^n(Y/S,\VV_Y)$.
%\end{center}
%\item \underline{products}: Assume $\EE$ to be a motivic \textit{ring} spectrum, then for any separated schemes 
%$\xymatrix{Y \ar[r]^q & X \ar[r]^p & S}$, for any integers $n,m$ and for any virtual vector bundles $\VV_X/X$ and $\mathcal{W}_Y/Y$, we have a multiplication map
%\begin{center}

%$\EE^m(Y/X,\mathcal{W}^\LL_Y)\otimes \EE^n(X/S,\VV^\LL_X)\to \EE^{m+n}(Y/S,\mathcal{W}^\LL_Y+\VV_Y)$.
%\end{center}
%\end{itemize}

%\begin{Par}
%{\sc base change} For any cartesian square
 % \begin{center}
%
 % $\xymatrix{
 % Y \ar[r]^g \ar[d]_q \ar@{}[rd]|-\Delta  & X \ar[d]^p \\
 % T \ar[r]_f & S,
 % }$
 % \end{center}
 % we have a map
 % \begin{center}
 % 
%$\Delta^*:\EE^n(X/S,\VV^\LL_X) \to \EE^n(Y/T,\VV^\LL_Y) $
%\end{center}
%where we recall that we denote by $\VV^\LL_X$ the virtual bundle %$\LL_{X/S}+\VV_X$.	
%
%\end{Par}

\begin{Par} {\sc Contravariance.} \label{contravarianceGeom} Let $f:Y\to X$ be a morphism of schemes and $\VV_X$ be a virtual bundle over $X$. There exists a pullback map
\begin{center}

$f^*:\EE^n(X,\VV_X)\to \EE^n(Y,\VV_Y)$.
\end{center}

\end{Par}

\begin{Par}{\sc Covariance.} \label{covarianceGeom}
Let $f:Y\to X$ be an lci projective map. There exists a Gysin morphism
\begin{center}

$f_*:\EE^n(Y,\LL_f+\VV_Y)\to \EE^n(X,\VV_X)$.

\end{center} As always, the definition follows from general considerations using the six functors formalism. For instance, assume $f$ smooth. By adjunction, the set $\Hom_{\SH(Y)}(\un_Y,\EE_Y\otimes \Tho (\LL_f))$ is in bijection with $\Hom_{\SH(X)}(\un_X,f_*(\EE_Y\otimes \Tho (\LL_f)))$. The purity isomorphism and the fact that $f_*=f_!$ (since $f$ is proper) lead a bijection with the set $\Hom_{\SH(X)}(\un_X,f_!f^!\EE_X)$. Any element of this set can be composed with the counit map $f_!f^!\to \Id$ so that we obtain an element in $\Hom_{\SH(X)}(\un_X,\EE_X)$.
\\ More generally, the group $\EE^r(Y,\LL_f+\VV_Y)$ is isomorphic to $[f^*(\un_X),f^*(\EE_X\otimes \Tho (\VV_X)[r])\otimes \Tho (\LL_f)]_Y$ which is, by adjunction, isomorphic to $[\un_X, f_*(f^*(\EE_X\otimes \Tho (\VV_X)[r])\otimes \Tho (\LL_f))]_X$. We can then compose with the trace map $\Tr_f:f_*\Sigma^{\LL_f}f^*\to \Id$ defined in \cite[§2.5.3]{DJK18} in order to obtain a map whose target is the group $\EE^r(X,\VV_X)$.
\end{Par}
\begin{Par} {\sc Milnor-Witt action.} \label{MWactionGeom}
Any motivic spectrum $\EE$ equipped with a unit isomorphism $\un \otimes \EE \to \EE$ defines an action by composition on the left
\begin{center}

$\gamma:H^m(X,\mathcal{W}_X)\otimes \EE^n(X,\VV_X)\to \EE^{m+n}(X,\mathcal{W}_X+\VV_X)$.
\end{center}
\end{Par}

\begin{Par} {\sc Localization long exact sequence.} \label{LocalizationMap}
Another essential property is the following long exact sequence deduced from the usual localization triangle of the six functors formalism in $\SH$.
\par Indeed, let $\EE\in \SH(S)$ be a motivic spectrum. For any closed immersion $i:Z\to X$ of separated schemes over $S$, with (quasi-compact) complementary open immersion $j:U\to X$, there exists a canonical localization long exact sequence of the form:
\begin{center}

$\xymatrixcolsep{1pc}\xymatrix{
\dots \ar[r] &
\EE^n(Z,\LL_Z+\VV_Z) \ar[r]^-{i_*} &
\EE^n(X,\LL_X+\VV_X) \ar[r]^-{j^*} &
\EE^{n}(U,\LL_U+\VV_U) \ar[r]^-{\partial_i} &
\EE^{n+1}(Z,\LL_Z+\VV_Z) \ar[r] &
\dots 
}$
\end{center}
where the residue map $\partial_i$ can be defined as the following composition
\begin{center}
$\xymatrix{
\EE^n(X-Z,\VV_{X-Z})\ar@{=}[r] \ar[d]^{\partial_i} & 
\EE_{-n}(X-Z/X-Z,-\VV_{X-Z}) \ar[r]^{\sim} &
 \EE_{-n}(X-Z/X,-\VV_{X-Z}) \ar[d]^{\partial} \\
\EE^{n+1}(Z,\LL_i+\VV_Z) &
\EE_{-n-1}(Z/Z,-\LL_i-\VV_Z) \ar@{=}[l] \ar[r]^{-\times \eta_i}_{\sim} &
\EE_{-n-1}(Z/X,-\VV_Z). 
}$
\end{center}
\end{Par}

\begin{Pro} \label{LocalizationGeom} 

Let $\EE\in \SH(S)$ be a motivic spectrum and consider the following commutative square
\begin{center}

$\xymatrix{
T \ar@{^{(}->}[r]^k  \ar@{^{(}->}[d]_q  & Y \ar@{^{(}->}[d]^p \\
Z \ar@{^{(}->}[r]_i & X 
}$

\end{center}
of closed immersions of separated schemes over $S$. For $\VV_X$ a virtual vector bundle over $X$, we have the following diagram
\begin{center}

$\xymatrixcolsep{1pc} \xymatrix{
\EE^n(T,*) \ar[r]^{k_*} \ar[d]_{q_*} &
\EE^n(Y,*) \ar[r]^{k'^*} \ar[r] \ar[d]^{p_*} &
\EE^n(Y- T, *) \ar[r]^{\partial_k} \ar[d]^{\tilde{p}_*} &
\EE^{n+1}(T,*) \ar[d]^{q_*} \\
\EE^n(Z,*) \ar[r]^{i_*} \ar[d]^{q'^*} & 
\EE^n(X,*) \ar[r]^{i'^*} \ar[d]^{p'_*} &
\EE^n(X- Z,*) \ar[d]^{\tilde{p}'^*} \ar[r]^{\partial_i} &
\EE^{n+1}(Z,*) \ar[d]^{q'^*} \\
\EE^n(Z- T,*) \ar[r]^{\tilde{i}_*} \ar[d]^{\partial_k} &
\EE^n(X- Y,*) \ar[r]^-{\tilde{i}'^*} \ar[d]^{\partial_{{p}}} &
\EE^n(X- (Z\cup Y), *) \ar@{}[rd]|-{(*)} \ar[d]^{\partial_{\tilde{p}}} \ar[r]^-{\partial_{\tilde{i}}} &
\EE^{n+1}(Z- T,*) \ar[d]^{\partial_{\tilde{k}}} \\
\EE^{n+1}(T,*) \ar[r]_{k_*} & 
\EE^{n+1}(Y,*) \ar[r]_{k'^*} &
\EE^{n+1}(Y- T,*) \ar[r]_{\partial_k} &
\EE^{n+2}(T,*)
}$
\end{center}
with obvious maps. Each square of this diagram is commutative except for $(*)$ which is anti-commutative.
\end{Pro}
\begin{proof}
This follows from Proposition \ref{LocalizationBiv}.
\end{proof}

\begin{Pro} \label{R3eGeom}

Let $\EE\in \SH(S)$ be a motivic spectrum. Let $i:Z\to X$ a closed immersion with complementary open immersion $j:U\to X$. Let $x\in H^m(Y,\mathcal{V}_Y)$. Then the following diagram is commutative:
\par 
%  \scalebox{0.9}{}

\begin{center}

$\xymatrixcolsep{1.2pc}\xymatrix{
\EE^{n}(Z,\LL_i+\mathcal{W}_Z) \ar[r]^{i_*} \ar[d]^{ \epsilon \circ \gamma_{i^*(x)}} &
\EE^n(X,\mathcal{W}_X) \ar[r]^{j^*} \ar[d]^{\gamma_{x}} &
\EE^n(U,\mathcal{W}_U) \ar[r]^{\partial_{Z,X}} \ar[d]^{\gamma_{j^*(x)}} &
\EE^{n-1}(Z,\LL_i+\mathcal{W}_Z) \ar[d]^{\epsilon \circ \gamma_{i^*(x)}} \\
\EE^{n+m}(Z,\LL_i+\VV_Z+\mathcal{W}_Z)  \ar[r]^{i_*}&
\EE^{n+m}(X,\VV_X+\mathcal{W}_X)  \ar[r]^{j^*}&
\EE^{n+m}(U,\VV_U+\mathcal{W}_U)  \ar[r]^-{\partial_{Z,X}} &
\EE^{n+m-1}(X,\LL_i+\VV_Z+\mathcal{W}_Z)
}$

\end{center}
where $\gamma_?$ is the multiplication map defined in \ref{MWactionGeom} (see also \ref{BivProduct}) and where $\epsilon$ is the isomorphism induced by the switch isomorphism $\LL_i + \VV_Z\simeq \VV_Z+\LL_i$.
\end{Pro}
\begin{proof}
This follows from \cite[Proposition 2.2.12]{DJK18}.
\end{proof}

\begin{Pro}[Base change for lci morphisms]\label{BaseChangelci} 
Consider a cartesian square of schemes
\begin{center}

$\xymatrix{
X' \ar[r]^{f'} \ar[d]_{g'} & Y' \ar[d]^g \\
X \ar[r]_f & Y
}$
\end{center}
with $f$ proper, and $g$ (quasi-projective) lci.
Suppose moreover that the square is {\em tor-independent}, that is for any $x\in X$, $y'\in Y'$ with $y=f(x)=g(y')$ and for any $i>0$ we have
\begin{center}

$\operatorname{Tor}^{\mathcal{O}_{Y,y}}_i(\mathcal{O}_{X,x},\mathcal{O}_{Y',y'})=0$.
\end{center}
Up to the canonical isomorphism $f^{'*}\LL_g \simeq \LL_{g'}$, we have:
\begin{center}
$f'_* \circ g'^* = g^*\circ f_*$.
\end{center}
\end{Pro}
\begin{proof}
See Proposition \ref{BaseChangeBiv}.
\end{proof}

We will need the following proposition:
\begin{Pro} \label{Prop.1.36}

Let $\nu:Z\to X$ a 
closed immersion of smooth schemes. Consider the canonical
 decomposition $Z=\sqcup_{i\in I} Z_i$ and $X=\sqcup_{j\in J} X_j$
  into connected components. Denote by $\hat{Z}_j=Z\times_X X_j$.
   For any $i\in I$, let $j\in J$ be the unique element such that $Z_i\subset X_j$ and denote by $\nu^j_i:Z_i\to Z_j$ the induced immersion. Consider the complement $Z_i'$ such that $\hat{Z}_i=Z_i\sqcup Z'_i$. The following diagram is commutative:
   \begin{center}
   
   $\xymatrixcolsep{3pc}\xymatrix{
   \EE^{n-1}(X-Z,*) \ar[r]^-{\partial_{X,Z}} \ar[d]^{\simeq} & 
   \EE^{n}(Z,*) \ar[r]^{\nu_*} \ar[d]^{\simeq} & 
   \EE^n(X,*) \ar[d]^{\simeq} \\
   \bigoplus_{j\in J} \EE^{n-1}(X_j-\hat{Z}_j,*) \ar[r]^-{(\partial_{ij})_{i\in I,j\in J}} &
   \bigoplus_{i\in I} \EE^n(Z_i,*) \ar[r]^{(\nu_{ij})_{i\in I,j\in J}} &
   \bigoplus_{j\in J} \EE^n(X_j,*)
   }
   $
   \end{center}
   where the vertical maps are the canonical isomorphisms and where, for any $(i,j)\in I\times J$, if $Z_i\subset X_j$, then $v_{ij}=(\nu^j_i)_*$ and $\partial_{ij}=\partial_{X_j-Z_j',Z_i}$ ; otherwise $v_{ij}=0$ and $\partial_{ij}=0$.
\end{Pro}

\begin{proof}
Straightforward.
\end{proof}

\section{From homotopy modules to Milnor-Witt cycle modules}
\label{HMandMW}

\subsection{Recollection on Milnor-Witt cycle modules}

We denote by $\mathfrak{F}_k$ the category whose objects are the couple $(E, \mathcal{V}_E)$ where $E$ is a field over $k$ and $\mathcal{V}_E\in \mathfrak{V}(E)$ is a virtual vector space (of finite dimension over $F$). A morphism $(E,\mathcal{V}_E)\to (F, \mathcal{V}_F)$ is the data of a morphism $E\to F$ of fields over $k$ and an isomorphism $\mathcal{V}_E \otimes_E F \simeq \mathcal{V}_F$ of virtual $F$-vector spaces.
\par A morphism $(E,\mathcal{V}_E)\to (F, \mathcal{V}_F)$ in $\mathfrak{F}_k$ is said to be finite (resp. separable) if the field extension $F/E$ is finite (resp. separable).
\par We recall that a Milnor-Witt cycle modules $M$ over $k$ is a functor from $\mathfrak{F}_k$ to the category $\Ab$ of abelian groups  equipped with data 
\begin{description}
\item \ref{itm:D1} (restriction maps) Let $\phi : (E,\mathcal{V}_E)\to (F, \mathcal{V}_F)$ be a morphism in $\mathfrak{F}_k$. The functor $M$ gives a morphism $\phi_*:M(E,\mathcal{V}_E) \to M(F,\mathcal{V}_F)$,
\item  \ref{itm:D2} (corestriction maps) Let $\phi : (E,\mathcal{V}_E)\to (F, \mathcal{V}_F)$ be a morphism in $\mathfrak{F}_k$ where the morphism $E\to F$ is {\em finite}. There is a morphism  $\phi^*:M(F,\Om_{F/k}+\mathcal{V}_F) \to M(E,\Om_{E/k}+\mathcal{V}_E)$,
\item \ref{itm:D3} (Milnor-Witt K-theory action) Let $(E,\mathcal{V}_E)$ and $(E,\mathcal{W}_E)$ be two objects of $\mathfrak{F}_k$. For any element $x$  of $\KMW(E,\mathcal{W}_E)$, there is a morphism 
\begin{center}
$\gamma_x : M(E,\mathcal{V}_E)\to M(E,\mathcal{W}_E+\mathcal{V}_E)$
\end{center}
so that the functor $M(E,-):\mathfrak{V}(E)\to \Ab$ is a left module over the lax monoidal functor $\KMW(E,-):\mathfrak{V}(E)\to \Ab$ (see \cite[Definition 39]{Yetter03} or \cite[Definition 3.1]{Fel18}),
\item \ref{itm:D4} (residue maps) Let $E$ be a field over $k$, let $v$ be a valuation on $E$ and let $\mathcal{V}$ be a virtual projective {$\mathcal{O}_v$-module} of finite type. Denote by $\mathcal{V}_E=\VV \otimes_{\mathcal{O}_v} E$ and $\VV_{\kappa(v)}=\VV \otimes_{\mathcal{O}_v} \kappa(v)$. There is a morphism
\begin{center}
$\partial_v : M(E,\VV_E) \to M(\kappa(v), - \NN_v+\VV_{\kappa(v)}),$
\end{center}
\end{description}   and satisfying rules 
\begin{description}
\item \ref{itm:R1a} (functoriality of restriction maps),
\item \ref{itm:R1b} (functoriality of corestriction maps),
\item \ref{itm:R1c} (base change property),
\item \ref{itm:R2a}, \ref{itm:R2b}, \ref{itm:R2c} (projection formulae),
\item \ref{itm:R3a} (ramification formula),
\item \ref{itm:R3b}, \ref{itm:R3c}, \ref{itm:R3d}, \ref{itm:R3e} (compatibility between residue maps and the first three data),
\item \ref{itm:R4a} (compatibility with orientations).
\end{description}

Moreover, a Milnor-Witt cycle module $M$ satisfies axioms \ref{itm:FD} (finite support of divisors) and \ref{itm:C} (closedness) that enable us to define a complex $(C_p(X,M,\VV_X),d_p)_{p\in \ZZ}$ for any scheme $X$ and virtual bundle $\VV_X$ over $X$ where
\begin{center}

$C_p(X,M,\VV_X)=\bigoplus_{x\in X_{(p)}}
 M(\kappa(x),\Omega_{\kappa(x)/k}+\VV_x)$
\end{center}
and the differential $d_p=(\partial^x_y)_{(x,y) \in X_{(p)}\times  X_{(p-1)}}	$ is defined as follows (see \cite[Section 4]{Fel18}).

\begin{Par} 
First, for $x$ a point of $X$, denote by 
\begin{center}
$M(x, \VV_X)=M(\kappa(x), \Om_{\kappa(x)/k}+\VV_x)$.
\end{center}
\par If $X$ is normal, then for any $x\in X^{(1)}$ the local ring of $X$ at $x$ is a valuation ring so that \ref{itm:D4} gives us a map $\partial_x: M(\xi, \VV_X) \to M(x, \VV_X)$ where $\xi$ is the generic point of $X$.

\label{2.0.1}

 If $X$ is any scheme, let $x,y$ be any points in $X$. We define a map
\begin{center}
$\partial^x_y:M(x,\VV_X) \to M(y,\VV_X)$
\end{center}
as follows. Let $Z=\overline{ \{x\}}$. If $y\not \in Z$, then put $\partial^x_y=0$. If $y\in Z$, let $\tilde{Z}\to Z$ be the normalization and put
\begin{center}
$\partial^x_y=\displaystyle \sum_{z|y} \cores_{\kappa(z)/\kappa(y)}\circ \, \partial_z$
\end{center}
with $z$ running through the finitely many points of $\tilde{Z}$ lying over $y$.

\end{Par}
\begin{Par} \label{FiveBasicMapsArticle2}
The complex $(C_p(X,M,\VV_X),d)_{p\geq 0}$ is called the {\em Milnor-Witt complex of cycles on $X$ with coefficients in $M$} and we denote by $A_p(X,M,\VV_X)$ the associated homology groups (called {\em Chow-Witt groups with coefficients in $M$}). We can define five basic maps on the complex level (see \cite[Section 4]{Fel18}):
\begin{description}
\item[Pushforward] Let $f:X\to Y$ be a $k$-morphism of schemes, let $\VV_Y$ be a virtual bundle over the scheme $Y$. The data \ref{itm:D2} induces a map
\begin{center}

$f_*:C_p(X,M,\VV_X)\to C_p(Y,M, \VV_Y)$.
\end{center}
\item[Pullback] Let $g:X\to Y$ be an essentially smooth morphism of schemes. Let $\VV_Y$ a virtual bundle over $Y$. Suppose $X$ connected (if $X$ is not connected, take the sum over each connected component) and denote by $s$ the relative dimension of $g$. The data \ref{itm:D1} induces a map
\begin{center}
$g^*:C_p(Y,M,\VV_Y) \to C_{p+s}(X,M,- \LL_{X/Y}+\VV_X)$.
\end{center}
 
 \item[Multiplication with units] Let $X$ be a scheme of finite type over $k$ with a virtual bundle $\VV_X$. Let $a_1,\dots, a_n$ be global units in $\mathcal{O}_X^*$. The data \ref{itm:D3} induces a map	
 \begin{center}
 $[a_1,\dots, a_n]:C_p(X,M,\VV_X) \to C_p(X,M,\ev{n}+\VV_X)$.
 \end{center}
 
 \item[Multiplication with $\eeta$]
 Let $X$ be a scheme of finite type over $k$ with a virtual bundle $\VV_X$. The Hopf map $\eeta$ and the data \ref{itm:D3} induces a map
 \begin{center}
 
 $\eeta:C_p(X,M,\VV_X)\to C_p(X,M,-\AAA^1_X+\VV_X)$.
 \end{center}
 
 \item[\namedlabel{itm:Bmaps}{Boundary map}]     
 Let $X$ be a scheme of finite type over $k$ with a virtual bundle $\VV_X$, let $i:Z\to X$ be a closed immersion and let $j:U=X\setminus Z \to X$ be the inclusion of the open complement. The data \ref{itm:D4} induces (as in \ref{2.0.1}) a map
 \begin{center}

 $\partial=\partial^U_Z:C_p(U,M,\VV_U) \to C_{p-1}(Z,M,\VV_Z)$.
 \end{center}

\end{description} 

 These maps satisfy the usual compatibility properties (see \cite[Section 5]{Fel18}). In particular, they induce maps $f_*,g^*, [u], \eeta, \partial^U_Z$ on the homology groups $A_*(X,M,*)$.
 
 \end{Par}

 \par We end this subsection with a lemma illustrating the importance of the rule \ref{itm:R4a}. This will be useful in Section \ref{EquivalenceOfCat}.
 \begin{Lem} \label{LemTrivialization}
  Let $M$ be a Milnor-Witt cycle module over $k$. For any field $E/k$ and any virtual vector bundle $\mathcal{V}_E$ over $E$, we have a canonical isomorphism
 \begin{center}
 
 $M(E,\VV_E)\simeq M(E,\langle n \rangle)\otimes_{\ZZ[E^{\times}]}\ZZ[\det(\VV_E)^{\times}]$
 \end{center}
 where $n$ is the rank of $\VV_E$.
  \end{Lem}
  \begin{proof}
  
  Any element $u\in \det(\VV_E)^{\times}$ defines an isomorphism
 \begin{center}
 
 $\Theta_u:M(E,\VV_E)\simeq M(E,\ev{n})\otimes_{\ZZ[E^{\times}]}\ZZ[\det(\VV_E)^{\times}]$
 \end{center}
 in a obvious way thanks to \ref{itm:D1}. One can check that this map does not depend on the choice of $u$ according to rule \ref{itm:R4a}.
  
  \end{proof}
  Note that this lemma is true for Milnor-Witt K-theory $\KMW$ by definition.

\subsection{Cycle premodule structure} \label{MWassociated}

Let $\EE\in \SH(S)$ be a motivic ring spectrum. For any field $E$ and any virtual vector space $\VV_E$ of rank $r$ over $E$, we put
\begin{center}

$\EEE(E,\VV_E)=\EE^{-r}(X,\VV_{X})=\Hom_{\SH(X)}(\un_X,\EE_X\otimes \Tho_X(\VV_X)[-r])$,
\end{center}
where $X=\Spec E$ (recall Definition \ref{DefBivTheory}).
We prove that this defines a functor $\EEE:\mathfrak{F}_k\to \Ab$ which is a Milnor-Witt cycle module. Indeed we have the following data:

\begin{description}
\item[\namedlabel{itm:D1}{(D1)}] Let $\phi : (E,\mathcal{V}_E)\to (F, \mathcal{V}_F)$ be a morphism in $\mathfrak{F}_k$. The cohomology theory $\EE^*(-,*)$ being contravariant (see \ref{contravarianceGeom}), we obtain a map $\phi_*:\EEE(E,\mathcal{V}_E) \to \EEE(F,\mathcal{V}_F)$.
\item [\namedlabel{itm:D2}{(D2)}] Let $\phi : (E,\mathcal{V}_E)\to (F, \mathcal{V}_F)$ be a morphism in $\mathfrak{F}_k$ where the morphism $E\to F$ is {\em finite}. The (twisted) covariance described in \ref{covarianceGeom} leads to a morphism  $\phi^*:\EEE(F,\Om_{F/k}+\mathcal{V}_F) \to \EEE(E,\Om_{E/k}+\mathcal{V}_E)$.
\item [\namedlabel{itm:D3}{(D3)}] Let $(E,\mathcal{V}_E)$ and $(E,\mathcal{W}_E)$ be two objects of $\mathfrak{F}_k$. For any element $x$  of $\KMW(E,\mathcal{W}_E)$, there is a morphism 
\begin{center}
$\gamma_x : \EEE(E,\mathcal{V}_E)\to \EEE(E,\mathcal{W}_E+\mathcal{V}_E)$
\end{center}
given by composition on the left by $x$ (as in \ref{MWactionGeom}) since we can identify $\KMW(E,\mathcal{W}_E)$ with $\hat{\un}(E,\mathcal{W}_E)$ (see Example \ref{MainExampleCoh}). We can check that the functor $\EEE(E,-):\mathfrak{V}(E)\to \Ab$ is then a left module over the lax monoidal functor $\KMW(E,-):\mathfrak{V}(E)\to \Ab$ (see \cite{Yetter03} Definition 39).

\item [\namedlabel{itm:D4}{(D4)}] Let $E$ be a field over $k$, let $v$ be a valuation on $E$ and let $\mathcal{V}$ be a virtual projective {$\mathcal{O}_v$-module} of finite type. As before, denote by $\mathcal{V}_E=\VV \otimes_{\mathcal{O}_v} E$ and $\VV_{\kappa(v)}=\VV \otimes_{\mathcal{O}_v} \kappa(v)$. There is a morphism
\begin{center}
$\partial_v : \EEE(E,\VV_E) \to \EEE(\kappa(v), - \NN_v+\VV_{\kappa(v)})$
\end{center}
given by the long exact sequence \ref{LocalizationMap} where the closed immersion is
\begin{center}

$\xymatrix{
\Spec \kappa(v) \ar@{^{(}->}[r] & \Spec \mathcal{O}_v
}$.
\end{center}

\end{description}

It is clear that the data \ref{itm:D1} and \ref{itm:D2} are functorial so that the two following rules hold:
\begin{description}\item [\namedlabel{itm:R1a}{(R1a)}] Let $\phi$ and $\psi$ be two composable morphisms in $\mathfrak{F}_k$. One has
\begin{center}
 $(\psi\circ \phi)_*=\psi_*\circ \phi_*$.
\end{center}
\item [\namedlabel{itm:R1b}{(R1b)}]  Let $\phi$ and $\psi$ be two composable finite morphisms in $\mathfrak{F}_k$. One has
\begin{center}
 $(\psi\circ \phi)^*=\phi^*\circ \psi^*$.
\end{center}

\end{description}

The base change theorem \ref{BaseChangelci} leads to the following rule \ref{itm:R1c}:
\begin{description}

\item [\namedlabel{itm:R1c}{(R1c)}] Consider $\phi:(E,\VV_E)\to (F,\VV_F)$ and $\psi:(E,\VV_E)\to (L,\VV_L)$ with $\phi$ finite and $\psi$ separable. Let $R$ be the ring $F\otimes_E L$. For each $p\in \Spec R$, let $\phi_p:(L,\VV_L)\to (R/p, \VV_{R/p})$ and $\psi_p:(F,\VV_F)\to (R/p, \VV_{R/p})$ be the morphisms induced by $\phi$ and $\psi$. One has
\begin{center}
$\psi_*\circ \phi^*=\displaystyle \sum_{p\in \Spec R} (\phi_p)^*\circ (\psi_p)_*$.
\end{center}

\end{description}

The general formalism of Fulton-McPherson gives the usual projection formulas (see also \cite[1.2.8]{DegBiv17}):
\begin{description}
\item [\namedlabel{itm:R2}{(R2)}] Let $\phi : (E,\VV_E)\to (F,\VV_F)$ be a morphism in $\mathfrak{F}_k$, let $x$ be in $\KMW (E,\mathcal{W}_E)$ and $y$ be in $\KMW (F,\Om_{F/k}+\mathcal{W'}_F)$ where $(E,\mathcal{W}_E)$ and $(F,\mathcal{W}'_F)$ are two objects of $\mathfrak{F}_k$.
\item [\namedlabel{itm:R2a}{(R2a)}] We have $\phi_* \circ \gamma_x= \gamma_{\phi_*(x)}\circ \phi_*$.
\item [\namedlabel{itm:R2b}{(R2b)}] Suppose $\phi$ finite. We have $\phi^*\circ \gamma_{\phi_*(x)}=\gamma_x \circ \phi^*$.
\item [\namedlabel{itm:R2c}{(R2c)}] Suppose $\phi$ finite. We have $\phi^*\circ \gamma_y \circ \phi_*= \gamma_{\phi^*(y)}$.

\end{description}

We now prove the remaining rules.

\begin{description}
\item	[\namedlabel{itm:R3a}{(R3a)}]
 
Let $E\to F$ be a field extension and $w$ be a valuation on $F$ which restricts to a non trivial valuation $v$ on $E$ with ramification index $e$. Let $\VV$ be a virtual {$\mathcal{O}_v$-module} so that we have a morphism ${\phi:(E,\VV_E)\to (F,\VV_F)}$ which induces a morphism
\begin{center}
 ${\overline{\phi}:(\kappa(v),-\NN_v+\VV_{\kappa(v)})\to ( \kappa(w),-\mathcal{N}_w+\VV_{\kappa(w)}})$. 
 
 \end{center}
 We have 
 
 \begin{center}
 
$\partial_w \circ \phi_*=\gamma_{e_\epsilon} \circ \overline{\phi}_* \circ \partial_v$. 
 \end{center}	 
\end{description}
\begin{proof} Consider the commutative diagram
\begin{center}

$\xymatrix{
\mathcal{O}_v \ar[r] \ar[d] & 
\mathcal{O}_w \ar[d] \ar@/^2pc/[ddr] & {} \\
   \kappa(v) \ar[r] \ar@/_2pc/[drr] &
   \kappa(v)\otimes_{\mathcal{O}_v} \mathcal{O}_w= \mathcal{O}_w/\mathfrak{m}_w^e \ar[rd]  & {} \\
                   &                   & \kappa(w).                 }$
\end{center}
which yields a topologically cartesian square
\begin{center}

$\xymatrix{
T \ar@{^{(}->}[r]^k  \ar[d]_q \ar@{}[rd]|-{\Delta} & Y \ar[d]^p \\
Z \ar@{^{(}->}[r]_i & X 
}$
\end{center}
where $k$ and $i$ are regular closed immersions of codimension 1 and $T=\Spec \kappa(w)$. Keeping the notations used in \ref{RamificationFormula}, the map $\nu^{(e)}$ induces a morphism of Thom spaces 
\begin{center}

$\Tho (N_TY)\to \Tho(q^*N_ZX)$
\end{center}
which corresponds after delooping to the quadratic form $e_\epsilon$ in $\kMW(\kappa(w))$ (see also \cite{Caz08}). We conclude thanks to Theorem \ref{RamificationGeom}.
\end{proof}

\begin{description}
\item [\namedlabel{itm:R3b}{(R3b)}]
Let $\phi: E\to F$ be a finite morphism of fields, let $v$ be a valuation over $E$ and let $\VV$ be a virtual vector bundle over $\mathcal{O}_v$. For each extension $w$ of $v$, we denote by $\phi_w:\kappa(v)\to \kappa(w)$ the map induced by $\phi$. We have
\begin{center}

$\partial_v\circ \phi^*=\sum_w \phi^*_w\circ \partial_w$.
\end{center}

\begin{comment}

where
\begin{center}
 $\phi^*:\EEE(F,\Om_{F/k}+\VV_F)\to \EEE(E,\Om_{E/k}+\VV_E)$

\end{center} and
\begin{center}

$\partial_v: \EEE(E,\Om_{E/k}+\VV_E) \to \EEE(E,\Om_{\kappa(v)/k}+\VV_{\kappa(v)}) $.

\end{center}

\end{comment}

\end{description}

\begin{proof}

There exists a semilocal ring $A$ over $\mathcal{O}_v$ such that the set of maximal ideals consists of the ideals $\mathfrak{m}_w$ where $w$ is an extension of $v$. 
\par Denote by $T=\oplus_{w|v} \Spec \kappa(w)$, $Y=\Spec A$, $Z=\Spec \kappa(v)$ and $X=\Spec \mathcal{O}_v$ so that we have the following commutative diagram formed by two cartesian squares:
\begin{center}

$\xymatrix{
T \ar@{^{(}->}[r]^k \ar[d]_g & Y \ar[d]^f & Y-T \ar@{_{(}->}[l]_{k'} \ar[d]^h \\
Z \ar@{^{(}->}[r]_i & X & X-Z \ar@{_{(}->}[l]^{i'}
}$
\end{center}
where $k,i$ are the canonical closed immersions with complementary open immersions $k',i'$ respectively and where $f,g,h$ are the canonical maps.
\par According to Proposition \ref{LocalizationGeom}, this leads to the following commutative diagram:
\begin{center}

$
\xymatrix{
\EE^r(T,\LL_T+\VV_T) \ar[r]^-{k_*} \ar[d]_{g_*} &
\EE^r(Y,\LL_Y+\VV_Y) \ar[r]^-{k'} \ar[d]^{f_*} &
\EE^r(Y-T,\LL_{Y-T}+\VV_{Y-T}) \ar[r]^-{\partial} \ar[d]^{h_*} \ar@{}[dr]|-{(*)} &
\EE^{r+1}(T,\LL_T+\VV_T) \ar[d]^{g_*} \\
\EE^r(Z,\LL_Z+\VV_Z) \ar[r]^-{i_*}  &
\EE^r(X,\LL_X+\VV_X) \ar[r]^-{i'}  &
\EE^r(X-Z,\LL_{X-Z}+\VV_{X-Z}) \ar[r]^-{\partial}  &
\EE^{r+1}(Z,\LL_Z+\VV_Z)  
}$
\end{center}
where $r$ is the rank of $\VV_X$. The rule \ref{itm:R3b} follows from the commutativity of the square $(*)$.

\end{proof}

\begin{description}
\item [\namedlabel{itm:R3c}{(R3c)}] Let $\phi : (E,\VV_E)\to (F,\VV_F)$ be a morphism in $\mathfrak{F}_k$ and let $w$ a valuation on $F$ which restricts to the trivial valuation on $E$. Then 
\begin{center}
$\partial_w \circ \phi_* =0$.
\end{center}

\end{description}

\begin{proof}
Consider the closed inclusion $i:Z\to X$ and its open complementary map $j:U\to X$ where $X=\Spec \mathcal{O}_v$, $Z=\Spec \kappa(w)$ and $U=\Spec F$. According to the long exact sequence \ref{LocalizationMap}, the composite
\begin{center}
$\xymatrix{
\EE^r(X,\VV_X) \ar[r]^{j^*} & 
\EE^r(U,\VV_U) \ar[r]^-{\partial_{Z,X}} & 
\EE^{r+1}(Z,\LL_{Z/X}+\VV_Z)
}$

\end{center}
is zero. The result follows from the fact that the map $\Spec E \to \Spec F$ factors through $j$ since $w$ restricts to the trivial valuation on $E$.

\end{proof}

\begin{description}

\item [\namedlabel{itm:R3d}{(R3d)}] Let $\phi$ and $w$ be as in \ref{itm:R3c}, and let $\overline{\phi}:(E,\VV_E)\to (\kappa(w),\VV_{\kappa(w)})$ be the induced morphism. For any prime $\pi$ of $v$, we have
\begin{center}
$\partial_w \circ \gamma_{[-\pi]}\circ \phi_*= \overline{\phi}_*$.
\end{center}

\end{description}
\begin{proof}

Denote by $Z=\Spec \kappa(w)$, $U=\Spec F$, $X=\Spec \mathcal{O}_w$ and $Y=\Spec E$ and consider the induced maps as defined in the following commutative diagram:
\begin{center}

$\xymatrix{
Z \ar@{^{(}->}[r]^i \ar[rd]_{\bar{f}} & 
X \ar[d]|-{\tilde{f}} &
U \ar@{_{(}->}[l]_j \ar[ld]^{{f}} \\
{} &
Y. &
{}
}$
\end{center}
We want to prove that the following diagram is commutative:

\begin{center}

$\xymatrix{
\EE^{-r}(Y,\VV_Y) \ar[d]_{\bar{f}} \ar[r]^{f^*}&
 \EE^{-r-1}(U,\VV_U) \ar[d]^{\gamma_{[-\pi]}} \\
\EE^{-r}(Z,\VV_Z)  &
\EE^{-r}(U,\AAA^1_U+\VV_U)    \ar[l]^-{\partial} 
}$

\end{center}
where we use the isomorphism $\NN_ZX\simeq \AAA^1_Z$ defined by the choice of prime $\pi$. We can split this diagram into the following one:
\begin{center}
$\xymatrix{
{} &&
\EE^{-r}(Y,\VV_Y) \ar[lld]_{\tilde{f}^*} \ar[d]^{h^*} \ar[rrd]^{f^*} &&
{} \\
\EE^{-r}(X,\VV_X)  \ar[dd]_{i^*} \ar[rr]^{q^*}  & \ar@{}_/-2pt/{(1)}[dd] &
\EE^{-r}(D_ZX - N_ZX,\VV) 
\ar@{}[rrd]|-{(2)}  \ar[d]|-{\gamma_{[-\pi]}} \ar[rr]^{d^*}
 &&
\EE^{-r}(U,\VV_U) \ar[d]^{\gamma_{[-\pi]}} \\
{} &&
\EE^{-r-1}(D_ZX-N_ZX,\AAA^1_U+\VV) 
\ar@{}[rrd]|-{(3)}  \ar[rr]^{d^*} \ar[d]^\partial &&
\EE^{-r-1}(U,\AAA^1_U+\VV_U) \ar[d]^\partial \\
\EE^{-r}(Z,\VV_Z)  &&
\EE^{-r}(N_ZX,\VV) \ar[ll]_{\simeq}^{p^*} \ar[rr]_{p^*}^{\simeq} &&
\EE^{-r}(Z,\VV)
}$
\end{center}
where the morphisms $d:U\to (D_ZX - N_ZX)$, $q:D_ZX - N_ZX\simeq \Gm \times X \to X$ and $p:N_ZX\to Z$ are the canonical maps used in the deformation to the normal cone.
\\ We can check that the square (1), (2) and (3) are commutative (same proof as \cite[Proposition 2.6.5]{Deg05}), hence the whole diagram is commutative by functoriality of the pullback maps.

\end{proof}

\begin{description}
\item [\namedlabel{itm:R3e}{(R3e)}] Let $E$ be a field over $k$, $v$ be a valuation on $E$ and $u$ be a unit of $v$. Then
\begin{center}
$\partial_v \circ \gamma_{[u]}=\gamma_{\epsilon[\overline{u}]} \circ \partial_v$ and
\\ $\partial_v \circ \gamma_{\eeta} =\gamma_{\eeta} \circ \partial_v$.
\end{center}

		\end{description}
		\begin{proof}
		This follows from Proposition \ref{R3eGeom} since $\epsilon \eeta = \eeta$ (where $\epsilon=-\ev{-1}$).

		\end{proof}
 
	\begin{description}
	\item [\namedlabel{itm:R4a}{(R4a)}] Let $(E,\VV_E)\in \mathfrak{F}_k$ and let $\Theta$ be an automorphism of $\VV$. Denote by $\Delta$ the canonical map from the group of automorphism of $\VV_E$ to the group $\kMW(E,0)$. Then
	 \begin{center}
	 
	 $\Theta_*=\gamma_{\Delta(\Theta)}:\EEE(E,\VV_E)\to \EEE(E,\VV_E)$.
	 \end{center}

	 \end{description}
	 
	 \begin{proof}
	 
	 One reduce to the case where $\EE=\un$. In this case, $\EEE$ is (isomorphic to) the Milnor-Witt K-theory $\KMW$ (see Example \ref{MainExampleCoh}), hence the result.
	 \end{proof}
	We have proved that $\EEE$ is a Milnor-Witt cycle premodule. In the following subsection, we prove that it satisfies axioms \ref{itm:FD} and \ref{itm:C}.
\subsection{Cycle module structure}
Put $S=\Spec(k)$. Let $f:X\to S$ be a scheme and $\VV_X$ be a virtual vector bundle over $X$. Let $\EE_S\in \SH(S)$ be a motivic spectrum. Recall that we denote by $\EE_X$ the spectrum $f^*(\EE_S)$. The purpose of this subsection is to prove that the Milnor-Witt cycle premodule $\EEE$ is in fact a cycle module. Roughly speaking, this means that the graded group $C^*(X,\EEE,*)$ forms a complex.
\par Consider a flag $\mathfrak{Z}=(Z_p)_{p\in \ZZ}$ over $X$, that is a sequence a closed subschemes of $X$ such that
\begin{center}

$\varnothing \subset Z_1 \subset Z_2 \subset \dots \subset Z_n \subset X $
\end{center}
where $\dim Z_p \leq p$.
\par For an integer $p\in \ZZ$, put $U_p=X-Z_p$ and $T_p=Z_p-Z_{p-1}$. Consider the canonical maps $j_p:U_p \subset U_{p-1}$ and $i_p:T_p\subset U_{p-1}$. 
\par For $p,q\in \ZZ$, denote by
\begin{center}

$E^{1,\mathfrak{Z}}_{p,q}=\EE^{q-p}(T_p,\LL_{T_p}+\VV_{T_p})
=[\mathbbm{1}_{T_p},\EE_{T_p}\otimes \Tho_{T_p}(\LL_{T_p}+\VV_{T_p})[q-p]]_{T_p} $
\end{center}
and
\begin{center}

$D^{1,\mathfrak{Z}}_{p,q}=\EE^{q-p-1}(U_p,\LL_{U_p}+\VV_{U_p})=
[\un_{U_p}, \EE_{U_p}\otimes \Tho_{U_p}(\LL_{U_p}+\VV_{U_p})[q-p-1]]_{U_p}.$
\end{center}
According to \ref{LocalizationMap}, we have a long exact sequence 

\begin{center}

$\xymatrix{
 \dots \ar[r] & D^{1,\mathfrak{Z}}_{p-1,q+1} \ar[r]^-{j_p^*} & D^{1,\mathfrak{Z}}_{p,q} \ar[r]^{\partial_p} & E^{1,\mathfrak{Z}}_{p,q}\ar[r]^{i_{p,*}} & D^{1,\mathfrak{Z}}_{p-1,q} \ar[r] & \dots
 }$
\end{center}
so that $(D^{1,\mathfrak{Z}}_{p,q},E^{1,\mathfrak{Z}}_{p,q})_{p,q\in \ZZ}$ is an exact couple. By the general theory (see \cite{McCleary00}, Chapter 3), this defines a spectral sequence. In particular, we have canonical differential maps $d$ which are well-defined and satisfying $d\circ d=0$. Moreover, we can prove that this spectral sequence converges to ${\EE^{p+q}(X,\LL_{X/S}+\VV_X)}$ (because the $E_{p,q}^1$-term is bounded) but we do not need this fact.
\par For $p,q\in \ZZ$, denote by
\begin{center}
$D_{p,q}^{1,X}=\colim_{\mathfrak{Z}\in \Flag(X)^{op}} D_{p,q}^{1,\mathfrak{Z}}$,\\
$E_{p,q}^{1,X}=\displaystyle \operatorname{colim}_{\mathfrak{Z}\in \Flag(X)^{op}} E^{1,\mathfrak{Z}}_{p,q}$
\end{center}
where the colimit is taken over the flags $\mathfrak{Z}$ of $X$. Since the colimit is filtered, we get an exact couple and a spectral sequence.

\begin{comment}
\begin{The} \label{SpectralSequenceHmtpInvariance}

We have the convergent spectral sequence
\begin{center}

$E_{p,q}^{1,X} \Rightarrow \EE^{p+q}(X,\VV^\LL_X)$.
\end{center}

\end{The}

\end{comment}

We need to compute this spectral sequence. This is done in the following.
\begin{The} \label{SpectralSequenceComputationBis}

For $p,q\in \ZZ$, we have a canonical isomorphism
\begin{center}

$E_{p,q}^{1,X}\simeq \displaystyle \bigoplus_{x\in X_{(p)}}[\un_{\kappa(x)},\EE_{\kappa(x)}\otimes \Tho_{\kappa(x)}(\LL_{\kappa(x)}+\VV_{\kappa(x)})[q-p]]_{\kappa(x)}$.
\end{center}
In particular, if $r$ is the rank of $\VV_X$, then
\begin{center}

$E_{p,-r}^{1,X}\simeq C_p(X,\EEE,\VV_X)$.
\end{center}
\end{The}
\begin{proof}
The proof is the same as \cite[Theorem 8.2]{Fel18}.
\end{proof}

\begin{The} \label{DifferentialsComputation}
Assume $X$ is a smooth scheme. Keeping the previous notations, the following diagram is commutative:
\begin{center}

$\xymatrix{
E^{1,X}_{p,q} \ar[r]^{d_{p,q}} \ar[d] & E^{1,X}_{p-1,q} \ar[d] \\
\EE^{q-p}(\Spec(\kappa(y),\LL_y+\VV_y) \ar[r]^{\partial^x_y} & \EE^{q-p+1}(\Spec \kappa(x),\LL_x+\VV_x)
}$
\end{center}
where $d_{p,q}$ is the differential canonically associated to the spectral sequence and where the vertical maps are the canonical projections associated to isomorphism of Theorem \ref{SpectralSequenceComputationBis}.
\end{The}
\begin{proof} (see also \cite[Proposition 1.15]{Deg12})

By definition, $d_{p,q}$ is the colimit of arrows
\begin{center}
$\xymatrix{
\EE^{q-p}(Z-Y,\LL_{Z-Y}+\VV_{Z-Y}) \ar[r]^-{i_{p*}} & \EE^{q-p}(X-Y,\LL_{X-Y}+\VV_{X-Y}) \ar[r]^-{\partial_{p-1}} & \EE^{q-p+1}(Y-W,\LL_{Y_W}+\VV_{Y_W})
}$
\end{center}
where $W\subset Y\subset Z$ are large enough closed subschemes with $\dim_X(Z)=p$, $\dim_X(Y)=p-1$ and $\dim_X(W)\leq p-2$. In the following, we consider $W,Y,Z$ as above. For simplification, we replace $X$ by $X-W$ so that we can remove any subset of $X$ if its dimension is $\leq p-2$.
\par Enlarging $Y$, we may assume that $Y$ contains $Z_{sing}$ the singular locus of $Z$. Since the singular locus of $Y$ has dimension strictly lesser than $p-1$, we may assume that $Y$ is smooth. In short, we study the composite:
\begin{center}

$\xymatrix{
\EE^{q-p}(Z-Y,\LL_{Z-Y}+\VV_{Z-Y}) \ar[r]^-{i_{Y*}} & \EE^{q-p}(X-Y,\LL_{X-Y}+\VV_{X-Y}) \ar[r]^-{\partial_p} & \EE^{q-p+1}(Y,\LL_Y+\VV_Y)
}$
\end{center} 
where $i_Y:Z-Y\to X-Y$ is the restriction of the canonical closed immersion $Z\to X$.
\par We denote by $Y_y$ (resp. $Z_x$) the irreducible component of $Y$ (resp. $Z$) containing $y$ (resp. $x$). We may decompose $Y$ as $Y=Y_y\sqcup  Y_y'$ with $Y_y'=Y\setminus Y_y$ since $Y$ is smooth. Denote by $\hat{Y}_x=Y\times_Z Z_x$ so that $Z_x-\hat{Y}_x$ is a connected component of the (smooth) scheme $Z-Y$. Denote by $i_x:Z_x-\hat{Y}_x \to  X-Y$ the canonical inclusion. According to Proposition \ref{Prop.1.36}, we get the following commutative diagram
\begin{center}

$\xymatrixcolsep{2pc}\xymatrix{
\EE^{q-p}(Z-Y,\LL_{Z-Y}+\VV_{Z-Y}) \ar[r]^-{i_{Y*}} \ar@{->>}[d] &
\EE^{q-p}(X-Y,\LL_{X-Y}+\VV_{X-Y}) \ar[r]^-{\partial_{X,Y}} \ar@{=}[d]& \EE^{q-p+1}(Y,\LL_Y+\VV_Y) \ar@{->>}[d] \\
\EE^{q-p}(Z_x-Y_x,\LL_{Z_x-Y_x}+\VV_{Z_x-Y_x}) \ar[r]^-{i_{x*}} 
\ar@/_2pc/[rr]_-{\partial^{Z,x}_{Y,y}} &
\EE^{q-p}(X-Y,\LL_{X-Y}+\VV_{X-Y}) \ar[r]^-{\partial_{X-Y'_y,Y_y}} & \EE^{q-p+1}(Y_y,\LL_{Y_y}+\VV_{Y_y})
}$

\end{center}
where the vertical maps are the canonical projections. The theorem is equivalent to proving that the differential $\partial^x_y$ defined in \ref{2.0.1} is the colimit of the maps $\partial^{Z,x}_{Y,y}$ defined in the above diagram.
\par Assume that $y$ is not a specialization of $x$, that is $y\not \in \overline{\{x\}}$. Then $\dim_X(Y_x\cap Z_x)\leq p-2$ hence (reducing $X$ to $X-(Y_x\cap Z_x)$) we may assume that $Y_x\cap Z_x=\varnothing$. Thus $Y_x\cap (Z_x- \hat{Y}_x)=\varnothing$ and we get the following cartesian square of closed immersions
\begin{center}

$\xymatrix{
\varnothing \ar[r] \ar[d] & Y_y \ar[d] \\
Z_x-\hat{Y}_x \ar[r] & X-Y_y' 
}$
\end{center}
which spawns the equality $\partial_{X-Y_y',Y_y} \circ i_{x*}$ by naturality of the residue maps (see Proposition \ref{LocalizationGeom}). This proves the proposition in this case.
\par Assume that $y$ is a specialization of $x$ so that $Y_y\subset Z_x$ and $Y_y\subset \hat{Y}_x$. For simplification, we assume that $Z=Z_x$, that is $Z$ is irreducible with generic point $x$. Consider the normalization $f:\tilde{Z}\to Z$ of $Z$. The singular locus $\tilde{Z}_{sing}$ is of codimension greater than $1$ in $\tilde{Z}$ hence $f(\tilde{Z}_{sing})$ is of dimension strictly lesser than $p-1$ in $X$ and (reducing $X$) we may assume that $\tilde{Z}$ is smooth.
\par Denote by $\tilde{Y}$ (resp. $\tilde{Y}_y$, $\tilde{Y}_y'$) the reduced inverse image of $Y$ (resp. $Y_y$, $Y'_y$) along $f$. Reducing $X$ again, we may assume that $\tilde{Y}_y$ is smooth and $\tilde{Y}_y\cap \tilde{Y}_y'=\varnothing$. We can also assume that every connected component of $\tilde{Y}_y$ dominates $Y_y$ (by reducing $X$, we can remove the non-dominant connected components). From this, we see that the map $g_y:\tilde{Y}_y\to Y_y$ induced by $f$ is finite and equidimensional. Consider the following topologically cartesian square:
\begin{center}

$\xymatrix{
\tilde{Y}_y \ar[r]^-{\tilde{\sigma}} \ar[d]^-{g_y} &
\tilde{Z}-\tilde{Y}'_y \ar[d] \\
Y_y \ar[r]^-{\sigma} & X-Y'_y
}$

\end{center}
where $\sigma$ and $\tilde{\sigma}$ are the canonical closed immersions and the right vertical map is induced by the composite 
\begin{center}

$\xymatrix{
\tilde{Z} \ar[r]^f & Z \ar[r]^i & X.
}$
\end{center}
By taking complements of $\tilde{\sigma}$ and $\sigma$, it induces the map
\begin{center}

$\xymatrix{
\tilde{Z}-\tilde{Y} \ar[r]^h & Z-Y \ar[r]^i & X-Y.
}$
\end{center}

By naturality of residues with respect to Gysin morphisms and by functoriality of the Gysin morphisms, we get the commutative diagram
\begin{center}

$\xymatrixcolsep{2pc}\xymatrix{
\EE^{q-p}(Z-Y,\LL_{Z-Y}+\VV_{Z-Y})  \ar[r]^-{i_{x*}} & \EE^{q-p}(X-Z,\LL_{X-Y}+\VV_{X-Y}) \ar[r]^-{\partial_{X-Y_y',Y_y}} & \EE^{q-p+1}(Y_y,\LL_{Y_y}+\VV_{Y_y}) \ar[d] \\
\EE^{q-p}(\tilde{Z}-\tilde{Y},\LL_{\tilde{Z}-\tilde{Y}}+\VV_{\tilde{Z}-\tilde{Y}}) \ar[u]^{h_*} \ar[r]_-{\partial_{\tilde{Z}-\tilde{Y}_y',\tilde{Y}_y}} &
\EE^{q-p}(\tilde{Y}_y,\LL_{\tilde{Y}_y}+\VV_{\tilde{Y}_y}) \ar[r]_-{g_{y*}} &
\EE^{q-p+1}(Y_y,\LL_{Y_y}+\VV_{Y_y}).
}$
\end{center}
For any $t\in f^{-1}(y)$, there exists a unique connected component $\tilde{Y}_t$ in the (smooth) scheme $\tilde{Y}_y$ so that $\tilde{Y}_y=\sqcup_{t\in f^{-1}(y)} \tilde{Y}_t$. Note that $\tilde{Y}_t$ is also a connected component of $\tilde{Y}$. Denote by $\tilde{Z}_t=\tilde{Z}-(\tilde{Y}-\tilde{Y}_t)$ ; this is an open subscheme of $\tilde{Z}$ containing $\tilde{Y}_t$ and $\tilde{Z}_t-\tilde{Y}_t=\tilde{Z}-\tilde{Y}$. According the Proposition \ref{Prop.1.36}, we have the following commutative diagram

\begin{center}
$\xymatrixcolsep{2pc} \xymatrix{
\EE^{q-p}(\tilde{Z}-\tilde{Y},\LL_{\tilde{Z}-\tilde{Y}}+\VV_{\tilde{Z}-\tilde{Y}}) \ar@{=}[d] \ar[r]^-{\partial_{\tilde{Z}-\tilde{Y}_y',\tilde{Y}_y}} &
\EE^{q-p+1}(\tilde{Y}_y,\LL_{\tilde{Y}_y}+\VV_{\tilde{Y}_y}) \ar[r]^-{g_{y*}} \ar[d]^{\simeq} &
\EE^{q-p+1}(Y_y,\LL_{Y_y}+\VV_{Y_y}) \ar@{=}[d] \\
\EE^{q-p}(\tilde{Z}-\tilde{Y},\LL_{\tilde{Z}-\tilde{Y}}+\VV_{\tilde{Z}-\tilde{Y}}) \ar[r]_-{\sum_t \partial_{\tilde{Z}_t,\tilde{Y}_t}} & 
\bigoplus_{t\in f^{-1}(y)} \EE^{q-p+1}(\tilde{Y}_t,\LL_{\tilde{Y}_t}+\VV_{\tilde{Y}_t}) 
\ar[r]_-{\sum_t g_{t*}} & 
\EE^{q-p+1}(Y_y,\LL_{Y_y}+\VV_{Y_y})
}$
\end{center}
where the middle vertical map is the canonical isomorphism.
\par We can now identify $\partial^x_y$ with the formal colimit of $\tilde{\partial}^{Z,x}_{Y,y}$ for $Y,W$. In view of \ref{2.0.1}, this is justified because:
\begin{itemize}
\item  $h$ is birational and $\tilde{Z}-\tilde{Y}$ is smooth with function field $\kappa(x)$.
\item The closed pair $(\tilde{Z}_t,\tilde{Y}_t)$ is smooth of codimension 1 and the local ring of $\mathcal{O}_{\tilde{Z}_t,\tilde{Y}_t}$ is isomorphic (through $h$) to the valuation ring $\mathcal{O}_{v_t}$ corresponding to the valuation $v_t$ on $\kappa(x)$ considered in \ref{2.0.1}.
\end{itemize}
\end{proof}
From Theorem \ref{SpectralSequenceComputationBis} and Theorem \ref{DifferentialsComputation}, we deduce that the differentials coincide so that $C_*(X,\EEE,\VV_X)$ is a complex when the scheme $X$ is smooth. We use this to prove that the premodule $\EEE$ is a Milnor-Witt cycle module:
\begin{description}
\item [\namedlabel{itm:FD}{(FD)}]  {\sc Finite support of divisors.}
	Let $X$ be a normal scheme, $\VV_X$ be a virtual vector bundle over $X$ and $\rho$ be an element of $M(\xi_X,\VV_X)$. Then $\partial_x(\rho)=0$ for all but finitely many $x\in X^{(1)}$.

\end{description}
\begin{proof}
 We can assume without loss of generality that $X$ is affine of finite ype. Then there exist a virtual vector bundle $\VV_{\AAA^r_k}$ over $\AAA^r_k$ and a closed immersion $i:X\to \AAA^r_k$ for some $r\geq 0$ which induces an inclusion
\begin{center}

$C^*(X,\EEE,\VV_X)\subset C^*(\AAA^r_k,\EEE, \VV_{\AAA^r_k})$
\end{center}
compatible with the differentials thanks to the previous theorem. Hence \ref{itm:FD} holds.

\end{proof}
\begin{description}
\item [\namedlabel{itm:C}{(C)}] {\sc Closedness.} Let $X$ be integral and local of dimension 2 and $\VV_X$ be a virtual bundle over $X$. Then
\begin{center}
$0=\displaystyle \sum_{x\in X^{(1)}} \partial^x_{x_0} \circ \partial^{\xi}_x: M(\xi_X,\VV_X)\to M(x_0, \VV_X)$
\end{center}
where $\xi$ is the generic point and $x_0$ the closed point of $X$.

\end{description}
\begin{proof}

 According to (FD), the differentials $d$ of $C^*(X,\EEE,\VV_X)$ are well-defined. We want to prove that $d\circ d=0$. Again, we can assume $X$ to be affine of finite type over $k$.  Then there exist a virtual vector bundle $\VV_{\AAA^r_k}$ over $\AAA^r_k$ and a closed immersion $i:X\to \AAA^r_k$ for some $r\geq 0$ which induces an inclusion
\begin{center}

$C^*(X,\EEE,\VV_X)\subset C^*(\AAA^r_k,\EEE, \VV_{\AAA^r_k})$
\end{center}
compatible with the differentials. Hence \ref{itm:C} holds.

\end{proof}
Since our constructions are natural in the motivic spectrum $\EE$, we can conclude:
\begin{The} \label{FunctorSHtoMW}
Consider $S=\Spec k$ the spectrum of a perfect field. The map $\EE \mapsto \EEE$ defines a functor from the category $\SH(S)$ of motivic spectrum to the category $\MW_k$ of Milnor-Witt cycle modules over $k$.
\end{The}

\section{An equivalence of categories} \label{EquivalenceOfCat}
The purpose of this section is to prove the following theorem.

\begin{comment}
A FAIRE: donner un plan de la preuve 
\end{comment}
\begin{The}\label{ThmDeg}
Let $k$ be a perfect field. The functor of Theorem \ref{FunctorSHtoMW} induces an equivalence between the category of Milnor-Witt cycle modules and the heart of Morel-Voevodsky stable homotopy category (equipped with the homotopy t-structure):
\begin{center}

$ \mathfrak{M}^{MW}_k \simeq {\operatorname{SH}(k)}^\heartsuit$.
\end{center}

\end{The}

\subsection{Associated homotopy module}
We recall some facts about the heart of the stable homotopy category (see \cite[§5.2]{Mor03} or \cite[§1]{Deg10}).

\begin{Def} \label{DefDesuspension}
	Let $M$ be an abelian Nisnevich sheaf on $\Sm$. We denote by $M_{-1}(X)$ the kernel of the morphism $M(X\times \Gm)\to M(X)$ induced by the unit section of $\Gm$. We say that $M$ is strictly homotopy invariant if the Nisnevich cohomology sheaf $H^*_{ \operatorname{Nis}}(-,M)$ is homotopy invariant.
\end{Def}
\begin{Def} \label{DefHomotopyModules}
	A {\em homotopy module} is a pair $(M_*,\omega_*)$ where $M_*$ is a $\ZZ$-graded abelian Nisnevich sheaf on $\Sm$ which is strictly homotopy invariant and $\omega_n:M_{n-1}\to (M_{n})_{-1}$ is an isomorphism (called {\em desuspension map}). A morphism of homotopy modules is a homogeneous natural transformation of $\ZZ$-graded sheaves compatible with the given isomorphisms. We denote by $\HM(k)$ the category of homotopy modules over $k$.
\end{Def}
\begin{Par}\label{MorelEquiv}
	For any spectrum $\EE$, the spectrum $\pizero(\EE)$ has a canonical structure of a homotopy module. Moreover, the functor $\pizero:\EE \mapsto \pizero(\EE)$	induces an equivalence of categories between the heart of $\SH(k)$ for the homotopy t-structure and the category $\HM(k)$ (see \cite{Mor03}). We denote its inverse by 
	 \begin{center}
	 	$H:\HM(k)\to \SH(k)^{\heartsuit}$.
	 \end{center}
\end{Par}

 \par We continue with two lemmas of independent interest.

\begin{Lem}

Let $g:Y\to X$ be a smooth morphism of schemes of finite type over $k$ of constant fiber dimension 1, let $\sigma$ be a section of $g$, let $\VV_X$ be a virtual vector bundle over $X$ and let $t\in \mathcal{O}_Y$ be a global parameter defining the subscheme $\sigma(X)$.
\par Then $\sigma_*:C_*(X,M,\VV_X)\to C_*(Y,M,\VV_Y)$ is zero on homology. 
\end{Lem}
\begin{proof}
Consider the open subscheme $j:U=Y\setminus \sigma(X)\to Y$ and let $\tilde{g}=g\circ j$ the restriction of $g$. Let $\partial$ be the boundary map associated to $\sigma$. According to \cite[Lemma 5.5]{Fel18}, we have $\sigma_*=\sigma_*\circ \partial \circ [t] \circ \tilde{g}^*=d\circ j_*\circ [t] \circ \tilde{g}^*$.

\end{proof}
With the same proof, we have a slightly more general result:
\begin{Lem} \label{LemmIntersNulle}
Let $g:Y\to X$, $\sigma:X\to Y$ and $\VV_X$ as previously. Let $i:Z\to X$ be a closed immersion and consider $\bar{Z}=g^{-1}(Z)$ the pullback
 along $g$. The induced map $\bar{\sigma}:Z\to \bar{Z}$ is such that the pushforward
 $\bar{\sigma}_*:C_*(Z,M,\VV_Z)\to C_*(\bar{Z},M,\VV_{\bar{Z}})$ is zero on homology.
\end{Lem}
\begin{Rem}
This result may be compared to \cite[Corollary 3.5]{FaselSrin08}.
\end{Rem}

Fix $M$ a Milnor-Witt cycle module over $k$. We associate to $M$ a homotopy module $F^M$, that is a homotopy invariant Nisnevich sheaf of $\ZZ$-graded abelian groups equipped with desuspension isomorphisms. Indeed, let $X$ be a smooth scheme over $S$. For any integer $n$, we put
\begin{center}
$F^M_n(X)=A^0(X,M,-\Om_{X/k}+\ev{n})$.

\end{center}
This defines a presheaf $F^M=(F^M_n)_{n\in \ZZ}$ of graded abelian groups satisfying the homotopy invariance property (see \cite[Theorem 8.3]{Fel18}).
\par Denote by $s_1:X=\{1\}\times X\to \Gm X$ the induced closed immersion. We have $(F^M_n)_{-1}(X)=\ker s^*_1$. By homotopy invariance, we have also $F^M_{n}(\AAA^1_X)\simeq F^M_n(X)$ hence $(F^M_n)_{-1}(X)=\coker j^*$ where $j$ is the open immersion $\Gm X\to \AAA^1_X$.
 \par As usual, we get the following long exact sequence:
 \begin{center}
 
 $\xymatrix{
 0 \ar[r] & F^M_{n}(X) \ar[r]^{j^*} &  F^M_{n}(\Gm X) \ar[r]^{\partial} & F^M_{n-1}(X) \ar[r]^-{i_*} & A^1(\AAA^1_X,M,-\Om_{\AAA^1_X/k}+\ev{n}).
 }$
 \end{center}
 Thus we see that $\partial$ induces a map ${(F^M_n)_{-1}(X)\to F^M_{n-1}(X)}$ which is an isomorphism because $i_*$ is zero (according to Lemma \ref{LemmIntersNulle}).
 \par We prove that $F^M$ is a Nisnevich sheaf. We start with the complex $C_*(-,M,*)$ and consider a Nisnevich square
 \begin{center}
 
 $\xymatrix{
 U_V \ar[r]^j \ar[d] & V \ar[d]^p \\
 U \ar[r]^i & X
 }$
 \end{center}
 where $i$ is open and $p$ étale. Denote by $Z=(X-U)_{\operatorname{red}}$ so that we have the decomposition 
 \begin{center}
 
 $C_*(X,M,*)=C_*(U,M,*)\oplus C_*(Z,M,*)$ 
 \end{center}
 and 
 \begin{center}
 
 $C_*(V,M,*)=C_*(U_V,M,*)\oplus C_*(Z_V,M,*)$. 
 \end{center}
 By assumption the induced map $p:Z_V\to Z$ is an isomorphism, hence the canonical map ${p_*:C_*(Z_V,M,*)\to C_*(Z,M,*)}$ is an isomorphism. Hence we can see that the image of the Nisnevich square by $C_*(-,M,*)$ is cocartesian. This proves that $C_*(-,M,*)$ is a Nisnevich sheaf and so is $F^M_*$.
 \par We have proved the following theorem.
 \begin{The} \label{HomotopyModuleFM}
 
 Let $M$ be a Milnor-Witt cycle module over $k$. The graded presheaf $F^M$ of abelian groups, defined by
 
 \begin{center}
 $F^M_n(X)=A^0(X,M,-\Om_{X/k}+\ev{n})$
 
 \end{center}
 for any smooth scheme $X/S$ and any integer $n$, is a homotopy module. 
 \end{The}

\subsection{First isomorphism} \label{FirstIsomArticle2}
In order to prove Theorem \ref{ThmDeg}, we construct two natural transformations and prove that they are isomorphisms. We start with the first isomorphism:
\par Let $M$ be a Milnor-Witt cycle module. Since the category of homotopy modules is equivalent to the heart of the stable homotopy category $\SH(S)$ (see \ref{MorelEquiv}), Theorem \ref{HomotopyModuleFM} implies that there is an object $\MM$ of $\SH(S)^{\heartsuit}$ equipped with isomorphisms 
\begin{center} $\label{EqFirstIso}
\alpha_X:\MM^{-n}(X,\langle n \rangle)\to F_n^M(X)$
\end{center}
for any irreducible smooth scheme $X$ of dimension $d$ and any integer $n$ (we recall that ${F_n^M(X)=A^0(X,M,-\Om_{X/k}+\langle n \rangle)}$). The maps $\alpha$ are compatible with the right-way maps (contravariance) and the desuspension functor $(-)_{-1}$ in the sense that the following diagrams commute
\begin{center}

$\xymatrix{
\MM^{-n}(X,\ev{n}) \ar[r]^{f^*} \ar[d]^{\alpha_X} & \MM^{-n}(Y,\ev{n}) \ar[d]^{\alpha_Y} \\
A^0(X,M,-\Om_{X/k}+\ev{n}) \ar[r]^{f^*} & A^0(Y,M,-\Om_{Y/k}+\ev{n})
}$
\end{center}
for any morphism $f:Y\to X$ of smooth schemes and
\begin{center}

$\xymatrix{
\MM^{-n}(X,\ev{n-1}) \ar[r]^{\omega_n} \ar[d]^{\alpha_X} & (\MM^{-n}(X,\ev{n}))_{-1} \ar[d]^{(\alpha_Y)_{-1}} \\
A^0(X,M,-\Om_{X/k}+\ev{n-1}) \ar[r]^{\omega'_n} & (A^0(X,M,-\Om_{X/k}+\ev{n}))_{-1}
}$
\end{center}
where $\omega_n$ and $\omega'_n$ are the structural desuspension maps associated the two homotopy modules for any integer $n$.
\par Fix $E/k$ a field and $n$ an integer. Using the previous isomorphism $\alpha_X$ with $X=\Spec A$ a smooth model of $E$ and taking the limit over all such $X$, we obtain an isomorphism of abelian groups
\begin{center}

$\alpha_{E}: \MMM(E,\ev{n})\to M(E,\ev{n})$.
\end{center} 
According to \ref{LemTrivialization}, this also defines in a canonical way an isomorphism
\begin{center}

$\alpha_{E}: \MMM(E,\VV_E)\to M(E,\VV_E)$
\end{center} 
for any virtual vector bundles $\VV_E$ over $E$.
\par We want to prove that this defines a morphism of Milnor-Witt cycle modules. It suffices to prove that $\alpha_E$ is natural in the data \ref{itm:D1}, \ref{itm:D2}, \ref{itm:D3} and \ref{itm:D4} (see \cite[Definition 3.5]{Fel18}).
% According to \ref{EqFirstIso}, it suffices to check that the five basic maps $f_*,g^*,[a],\eeta,\partial$ defined for $A^0(X,M,*)$ (\ref{Fel18}, Section ?) agree with the definitions of pushfoward $f_*$, pullbacks $g^*$, $\KMW$-action and residue $\partial$ defined for the bivariant theory $\MM^*(X,M,*)$.
\paragraph{(D1)}
For any morphism $f:Y\to X$ of smooth schemes, the maps $\alpha$ are compatible with right-way (pullbacks) morphisms thus the following diagram is commutative
\begin{center}
$\xymatrix{
\MMM(E,\VV_E) \ar[r]^{\res_{F/E}} \ar[d]^{\alpha_E} & \MMM(F,\VV_F) \ar[d]^{\alpha_F} \\
M(E,\VV_E) \ar[r]^{\res_{F/E}}& M(F,\VV_F)
}$

\end{center}
where $F/E$ is a field extension and $\VV_E$ is a virtual vector bundle over $E$.

\paragraph{(D4)}
 Let $Z$ be a smooth scheme over $S$. Since the maps $\alpha$ commute with the functor $(-)_{-1}$, we have the following commutative diagram (see \cite[Proposition 3.9]{Fel18})
 \begin{center}
 
 $\xymatrix{
 \MM^{-n}(\AAA^1_Z,\ev{n}) \ar[r]^{j^*} \ar[d]^\alpha & \MM^{-n}(\Gm Z,\ev{n}) \ar[d]^{\alpha} \ar[r]^{\partial}& \MM^{-n}(Z,\ev{n}) \ar[d]^\alpha \\
 A^0(\AAA^1_Z,-\Om_{\AAA^1_Z/S}+\ev{n}) \ar[r]^{j^*}  & A^0(\Gm Z,-\Om_{\Gm Z/S}+\ev{n}) \ar[r]^-{\partial} & A^0(Z,-\Om_{Z/S}+\ev{n}) 
 }
 $ 
 \end{center}
 where $j:\Gm Z \to \AAA^1_Z$ is the open immersion complementary to the zero section $i:Z\to \AAA^1_Z$.
 \par By deformation to the normal cone, we have the same commutative diagram when $j:\Gm Z \to \AAA^1_Z$ is replaced by an open immersion $j:X-Z\to X$ associated with a regular immersion $i:Z\to X$ of codimension 1. In particular, when $X=\Spec \mathcal{O}_v$ is the spectrum of a valuation ring and $Z=\Spec \kappa(v)$, we find that the maps $\alpha$ are compatible with the residue maps:
 \begin{center}
 
 $\xymatrix{
 M(E,\VV_E) \ar[r]^-{\partial_v} \ar[d]_-{\alpha_E} & M(\kappa(v),-\NN_v+\VV_{\kappa(v)}) \ar[d]^-{\alpha_{\kappa(v)}} \\
 \MMM(E,\VV_E) \ar[r]_-{\partial_v} & \MMM(\kappa(v),-\NN_v+\VV_{\kappa(v)}) 
 }$
 \end{center}
 is a commutative square.
 
 \paragraph{(D2)} Let $E$ be a field. The homotopy invariance property \ref{itm:H} states that the following sequence is split exact:
   \begin{description}
     \item [\namedlabel{itm:H}{(H)}]
     $
     \xymatrix{
     0 \ar[r] &  \MMM(E,\AAA^1_E+\Om_{E/k}+\VV_E) \ar[r]^{\res_{E(t)/E}}   & \MMM(E(t),\Om_{E(u)/k}+ \VV_{E(u)})  \\
       & \, \, \,\, \, \,\, \, \,\, \, \,\, \, \, \, \, \, \ar[r]^-d
      & \bigoplus_{x\in {(\AAA_E^1)}^{(1)}} \MMM(\kappa(x),\Om_{\kappa(x)/k}+ \VV_{\kappa(x)}) \ar[r] & 0
     }
     $
     \end{description}

  where $d=\sum_{x\in {(\AAA_E^1)}^{(1)}}\partial_x$ and where $\VV_E$ is a virtual vector bundle over $E$ (this is true for any Milnor-Witt cycle module hence in particular for $\MMM$). 
  \par We can use this property \ref{itm:H} and the data \ref{itm:D1} to characterize the data \ref{itm:D2}. Indeed, let $F/E$ be a finite field extension. Assume $F/E$ is monogenous, thus $F=E(x)$ where $x$ corresponds to a point in ${(\AAA_F^1)}^{(1)}$. For any $\beta\in M(F,\Om_{F/k}+\VV_{F})$ there exists ${\gamma\in M(E(t),\Omega_{E(t)/k}+\VV_{E(t)})}$ with the property that $d(\gamma)=\beta$. Now the valuation at $\infty$ yields a morphism
     \begin{center}
     ${\partial_\infty:M(E(t),\Omega_{E(t)/k}+\VV_{E(t)})\to
      M(E,\Omega_{E/k}+\VV_E)}$ 
     \end{center}which vanishes on the image of $\res_{E(t)/E}$. The element $-\partial(\gamma)$ does not depend on the choice of $\gamma$ and is in fact equal to $\cores_{F/E}(\beta)$. Using this characterization, we see that \ref{itm:D2} commutes with the maps $\alpha$ since they commute with \ref{itm:D4}.

\paragraph{(D3)} In order to prove that the maps $\alpha$ commute with the $\KMW$-action on the left, it suffices to do it for any generator $[u]$ (where $u$ is a unit) and the Hopf map $\eeta$. 
\par Let $E$ be a field over $k$ and $u$ be a unit of $E$. Denote by $X$ the essentially smooth scheme $\Spec E$. The unit defines a map $u:X\to \Gm X$ which induces a map
\begin{center}

$u^*:\MM^{-n}(\Gm X,\ev{n}) \to \MM^{-n}(X,\ev{n})$
\end{center} 
for any integer $n$. Moreover, we consider the canonical maps
\begin{center}

$\omega _n:\MM^{-(n-1)}(X,\ev{-1+n})\to (\MM^{-n}(X,\ev{n}))_{-1}$ 
\end{center}
and
\begin{center}

$\nu_n:(\MM^{-n}(X,\ev{n}))_{-1}\subset  \MM^{-n}(\Gm X,\ev{n})$.
\end{center}
Now consider the canonical morphism
\begin{center}

$\eeta^*:\MM^{-n}(X,\ev{n})\to \MM^{-n}(\Gm X,\ev{n})$
\end{center}
induced by the Hopf map and the canonical projection
\begin{center}

$\pi_n: \MM^{-n}(\Gm X,\ev{n})\to (\MM^{-n}(X,\ev{n}))_{-1}$.
\end{center}
One can check that the data D3 satisfies
\begin{center}

$\gamma_{[u]}=u^*\nu_n\omega_n:\MM^{-(n-1)}(X,\ev{-1+n})\to \MM^n(X,\ev{n})$.
\end{center}
and
\begin{center}

$\gamma_{\eeta}=\omega_n\pi_n\eeta^*: \MM^n(X,\ev{n})\to \MM^{-(n-1)}(X,\ev{-1+n})$
\end{center}
We have the same description for the Milnor-Witt cycle module $M$. Since the maps $\alpha$ commute with pullbacks and transition maps $\omega_n$, we see that they also commute with the $\KMW$-action.

\subsection{Second isomorphism}
Let $\MM\in \SH(k)^\heartsuit$. Let $X$ be a smooth scheme over $k$ and let $x$ be a generic point of $X$. For any integer $n$, we have a canonical map $\MM^{-n}(X,\ev{n})\to \MM^{-n}(\kappa(x),\ev{n})=\MMM(\kappa(x),\ev{n})$. Thus we have a map $\beta_X:\MM^{-n}(X,\ev{n})\to C^0(X,\MMM,-\Om_{X/k}+\ev{n})$ which factors through $A^0(X,\MMM,-\Om_{X/k}+\ev{n})$. We want to prove that the arrow
\begin{center}

$\beta:\MM\to A^0(-,\MMM,*)$
\end{center}
is an isomorphism of homotopy modules.
\par We prove that $\beta_X$ are natural in $X$ (with respect to Gysin morphisms). If $p:Y\to X$ is a smooth map of smooth schemes, it is clear by the definition of pullbacks for Chow-Witt groups with coefficients in $M$ (see \cite[§4.5]{Fel18}) that $p^*$ commutes with $b$.
\par Now consider a regular closed immersion $i:Z\to X$ of smooth schemes. Recall that (by \cite[Definition 9.1]{Fel18}) the Gysin morphism $i^*$ (for Chow-Witt groups with coefficients in $M$) makes the following diagram commutative
\begin{center}

$\xymatrix{
A^0(X,M,\ev{n}) \ar[r]^-{q^*} \ar@{=}[d]& A^0(\Gm X,M,-\LL_q+\ev{n}) \ar[r]^-{[t]}& A^0(\Gm X,M,\ev{n}) \ar[d]^{\partial} \\
A^0(X,M,\ev{n})\ar[r]^-{i^*} & A^0(Z,M,-\LL_i+\ev{n}) \ar[r]^{\pi^*}_{\simeq} & A^0(N_ZX,M,\ev{n})
}$
\end{center} 
where $q:\Gm X\to X$ is the canonical projection and $t$ is a parameter such that $\AAA^1_k=\Spec k[t]$ and where $\LL_i=-\NN_ZX$.
\par Similarly, the Gysin morphism $i^*$ (for the cohomology theory $\MM$) makes the following diagram commutative
\begin{center}

$\xymatrix{
\MM^{-n}(X,\ev{n}) \ar[r]^-{q^*} \ar@{=}[d]& \MM^{-n}(\Gm X,\ev{n}) \ar[r]^-{[t]}& \MM^{-n-1}(\Gm X,\ev{n+1}) \ar[d]^{\partial} \\
\MM^{-n}(X,\ev{n}) \ar[r]^-{i^*} & \MM^{-n}(Z,\ev{n}) \ar[r]^{\pi^*}_{\simeq} & \MM^{-n}(N_ZX,\ev{n})
}$
\end{center} 
where $q,\pi$ and $t$ are defined as previously (the proof is the same as \cite[Proposition 2.6.5]{Deg05}). Putting things together, we see that the maps $\beta$ commute with $i^*$ hence with any pullbacks (of lci morphisms).
\par Moreover, we prove that $\beta$ is compatible with the desuspension maps $\omega_n:\MM_{n-1}\simeq (\MM_n)_{-1}$ defining the homotopy modules $\MM$ and $A^0(-,\MMM,*)$. Let $X$ be an irreducible smooth scheme, we have the following diagram:
\begin{center}

$\xymatrix{
0 \ar[r]  & \MM^{-n}(\AAA^1_X,\ev{n}) \ar[r]^{j^*} \ar[d]^{\beta} \ar@{}[rd]|-{(1)} & \MM^{-n}(\Gm X,\ev{n}) \ar[r]^-{\partial_\MM} \ar[d]^{\beta} \ar@{}[rd]|-{(2)} & \MM^{-n+1}(X,\ev{n-1}) \ar[r] \ar[d]^{\beta} & \dots 
\\
0 \ar[r] & A^0(\AAA^1_X,\MMM,*) \ar[r]^{j^*} & A^0(\Gm X,\MMM,*) \ar[r]^{\partial} & A^0(X,\MMM,*-1) \ar[r] & \dots . 
}$
\end{center}
We have already seen that the square (1) commutes. The map $\partial$ is defined in \ref{FiveBasicMapsArticle2} using the data \ref{itm:D4} of the cycle module $\MMM$ which corresponds to the map $\partial_\MM$. Hence the square (2) commutes. According to Definition \ref{DefDesuspension}, the desuspension map $\omega_n:\MM_{n-1}\simeq (\MM_n)_{-1}$ is induced by $\partial_\MM$. Thus ${\beta}$ is a morphism of homotopy modules.
\par Finally, when $X$ is the spectrum of a field, the map ${\beta}_X$ is an isomorphism and so ${\beta}$ is an isomorphism of homotopy modules.

\par 
Putting the second isomorphism $\beta$ with the first isomorphism $\alpha$ of Subsection \ref{FirstIsomArticle2}, we have proved Theorem \ref{ThmDeg}.

\section{Applications} \label{Applications}

\subsection{Hermitian K-theory and Witt groups}

We assume that the characteristic of $k$ is different from $2$.
\par In \cite{Ati66}, Atiyah studied the topological K-theory of $\ZZ/2$-bundles on spaces with involution, expanding what we knew about real topological K-theory. The algebraic analogue is called {\em Hermitian K-theory} and was first introduced by Karoubi (see e.g. \cite{Kar80, Kar80bis}). A natural question was to translate this notion into the work of Morel and Voevodsky.
\par In \cite{Hor05}, Hornbostel proved that hermitian K-theory is representable in the stable homotopy category of Morel and Voevodsky. Precisely, there is a motivic $(8,4)$-periodic spectrum representing hermitian K-theory over the field $k$.  
\par Moreover, the theory of quadratic forms was studied by Balmer. In particular, he introduced a graded $4$-periodic generalization $W^*_B$ of Witt groups (with the classical Witt groups standing in degree $0$, see \cite{Bal00,Bal01} for more details). Similarly, Hornbostel proved that there is a spectrum whose homotopy groups coincide with the groups $W^*_B$.
\par Thus, according to Theorem \ref{ThmDeg},  we have the following theorem.
\begin{The}\label{KHisModule} 
There exist Milnor-Witt cycle modules $\KO$ and $\KW$ respectively associated to Hermitian K-theory and Balmer Witt groups in a canonical way.
\end{The}

\subsection{Monoidal structure, adjunction and equivalences of categories}
Recall that a Grothendieck category is an abelian category $\CCC$ with (infinite) coproducts (hence, all colimits) such that filtered colimits of exact sequences are exact and admitting a generator, that is, an object $G\in \CCC$ such that the functor $\Hom_{\CCC}(G,-)$ is faithful.
\par Thanks to Theorem \ref{ThmDeg}, we can transpose known properties from the category of homotopy modules to the category of Milnor-Witt cycle modules:

\begin{The}
	The category $\MW_k$ of Milnor-Witt cycle modules is a Grothendieck category with products. Moreover, there is a canonical symmetric closed monoidal structure on $\MW_k$ such that the unit element is the cycle module $\KMW$. In addition, the monoidal tensor product commutes with the shifting functor defined in \cite[Example 4.7]{Fel18}. 
\end{The}

Using the theory of framed correspondences, Ananyevskiy and Neshitov constructed Milnor-Witt transfers on homotopy modules, proving this way that the hearts of the homotopy t-structures on the stable $\AAA^1$-derived category and the category of Milnor-Witt motives are equivalent \cite{Neshitov2018}. Assuming $k$ to be an infinite perfect field of characteristic not two, their proof relies on the work of Garkusha and Panin \cite{GarkushaPanin14,GarkushaPanin18}. Similarly, one could give another proof of this fact:

\begin{The} \label{ThmAnaNeshi}
The category of Milnor-Witt cycle modules is equivalent to the heart of the category of MW-motives (equipped with the homotopy t-structure):
\begin{center}

$ \mathfrak{M}^{MW}_k \simeq \DMt(k)^\heartsuit$.
\end{center}
In particular, the heart of Morel-Voevodsky stable homotopy category is equivalent to the heart of the category of MW-motives \cite{DegFas18} (both equipped with their respective homotopy t-structures):
\begin{center}
$\SH(k)^\heartsuit \simeq \DMt(k)^\heartsuit$.

\end{center}

\end{The}
\begin{proof} Let $M$ be a Milnor-Witt cycle module. It corresponds to a homotopy module $F^M$ according to the previous section. Since we have an action of the Milnor-Witt K-theory on $F^M$, we can prove (as in \cite[§3.2]{Fel20}) that $F^M$ has in fact MW-transfers (see \cite[§6.1]{Fel20} for the definition of homotopy modules with MW-transfers). We can then proceed as in the proof of Theorem \ref{ThmDeg} to prove the first equivalence of categories. The second equivalence follows from Theorem \ref{ThmDeg}.
\end{proof}

\begin{comment}
\begin{The}

Let $n\geq 1$. Then the canonical functor
\begin{center}

${\Sigma^{\infty-n}_{\Gm}}^{\heartsuit}:{\operatorname{SH}^{S^1}(k)(n)}^{\heartsuit}\to {\operatorname{SH}^{eff}(k)}^{\heartsuit}$
\end{center}
is an equivalence of abelian categories.
\end{The}
 \begin{proof}
 	
This equivalent to proving that, for any smooth schemes $X/k$, the presheaf $\pizero^{\mathbb{S}^1}(\Gm^nX_+)$ is a homotopy module (in the image of $\omega^\infty$)
 \end{proof}

\end{comment}
In his thesis \cite{Deg03}, Déglise studied  the category of homotopy modules with transfers which is known to be equivalent to the heart of the category of Voevodsky’s motives $\DM(k,\ZZ)$ (with respect to the homotopy t-structure). Déglise's main theorem was that this category can be described with Rost's theory of cycle modules. This fact could be rediscovered thanks to our previous results:
\begin{The}[Déglise]\label{theseDeglise}
Let $k$ be a perfect field. The category of Rost cycle modules over $k$ is equivalent to the heart of the category of Voevodsky’s motives $\DM(k,\ZZ)$ with respect to the homotopy t-structure:
\begin{center}
$\mathfrak{M}^{\operatorname{M}}_k \simeq\DM(k,\ZZ)^\heartsuit.$
\end{center}
\end{The}
\begin{proof}
One can see that the category of Rost cycle modules is equivalent to the full subcategory of $\MW_k$ of Milnor-Witt cycle modules with trivial action of the generator $\eeta$. Thanks to Theorem \ref{ThmAnaNeshi}, this subcategory is equivalent to the full subcategory of $\DMt(k)^\heartsuit$ of homotopy modules with transfers and with trivial action of the Hopf map $\eeta$. This last category is equivalent to $\DM(k,\ZZ)^\heartsuit$.
\end{proof}

\paragraph{Adjunction between MW-cycle modules and Rost cycle modules}
\par  Consider $M$ a classical cycle module (\textit{à la Rost}, see \cite[§1]{Rost96}). Recall that we may define a Milnor-Witt cycle module $\Gamma_*(M)$ as follows. Let $(E,\VV_E)$ be in $\mathfrak{F}_k$ and put
\begin{center}

$\Gamma_*(M)(E,\VV_E)=M(E,\rk \VV_E)$.
\end{center}
We can check that this defines a fully faithful exact functor
\begin{center}

$\Gamma_*:\mathfrak{M}^{\operatorname{M}}_k\to \mathfrak{M}^{\operatorname{MW}}_k$.
\end{center}
where $\mathfrak{M}^{\operatorname{M}}_k$ (resp. $\mathfrak{M}^{\operatorname{MW}}_k$) is the category of Rost cycle modules (resp. Milnor-Witt cycle modules). This definition leads to the following theorem:

\begin{The}[Adjunction Theorem]\label{AdjunctionTheoremBis}

There is an adjunction between the category of Milnor-Witt cycle modules and the category of classical cycle modules:
\begin{center}
$\mathfrak{M}^{\operatorname{MW}}_k \rightleftarrows \mathfrak{M}^{\operatorname{M}}_k$.

\end{center}
\end{The}
\begin{proof}
We gave an elementary proof of this result in \cite[Section 12]{Fel18} with an explicit description of the adjoint functors. For a second proof, combine Theorem \ref{ThmDeg} and Theorem \ref{theseDeglise}.
\end{proof}

\subsection{Birational invariance}

Studying unramified cohomology groups with $\ZZ/2$-coefficients, one can see that an elliptic curve is not birational to the projective line. More generally, étale cohomology and K-theory are a source of such birational invariants (see \cite{Col92} for more details).

\par Rost proved in \cite[Corollary 12.10]{Rost96} that, if $X$ is a proper smooth variety over $k$ and $M$ a cycle module, then the group $A^0(X,M)$ is a birational invariant of $X$. A natural question is to extend this for Milnor-Witt modules, hoping that the quadratic nature of our theory will lead to more refined (birational) invariants and thus sharper theorems. Rost's proof heavily depends on the existence of pullback maps for flat morphisms. Unfortunately, such pullback maps remain to be constructed in our setting. Nevertheless, a different method yields the expected result:

\begin{comment}
A FAIRE : suggestion : trouver un contre-exemple au theorème ie un spectre E (non dans le coeur de SH avec E^00(X) non-invariant birationnel.)
\end{comment}
\begin{The} \label{BirInv}
	Let $X$ be a proper smooth integral scheme over $k$, let $\VV_k$ a virtual vector bundle over $k$ and let $M$ be a Milnor-Witt cycle module. Then the group $A^0(X,M,-\Om_{X/k}+\VV_X)$ is a birational invariant of $X$ in the sense that, if $X\dashrightarrow Y $ is a birational map, then there is an isomorphism of abelian groups
\begin{center}
$A^0(Y,M,-\Om_{Y/k}+\VV_Y) \to A^0(X,M,-\Om_{X/k}+\VV_X)$.
\end{center}
In particular for $M=\KMW$, we obtain the  fact that the Milnor-Witt K-theory groups $\kMW_n$ are birational invariants.
\end{The}
\begin{proof} (see also \cite[Lemma 1.3]{Voi19}).
	Denote by $F^M(X)=A^0(X,M,-\Om_{X/k}+\VV_X)$. This defines a contravariant functor that satisfies:
	\begin{enumerate}
	\item If $U\subset X$ is a Zariski open set, then the map $F^M(X)\to F^M(U)$ is injective,
\item 	If $U\subset X$ is a Zariski open set such that $\codim_X(X\setminus U)\geq 2$, then the map $F^M(X)\to F^M(U)$ is an isomorphism.\end{enumerate}
Indeed, these properties follow from the localization long exact sequence \cite[§6.4]{Fel18}.
\par Now let $\Phi: X \dashrightarrow Y $ be a birational map between smooth and proper integral schemes over $k$. Then there is an open set $U\subset X$ such that $\codim_X(X\setminus U)\geq 2$ and $\Phi_U$ is an morphism. Then we have $F^M(X)\simeq F^M(U)$ and, by functoriality, a morphism $\Phi_U^*:F^M(Y)\to F^M(U)$, hence a morphism $\Phi_*:F^M(Y)\to F^M(X)$. Replacing $\Phi$ by $\Phi^{-1}$, we get $\Phi_V^{-1}:F^M(X)\to F^M(V)$ for some Zariski open set $V$ of $Y$ such that $F^M(Y)\simeq F^M(V)$. Let $U'\subset U$ be defined as $\Phi_V^{-1}(V)$. Then $\Phi^{-1}\circ \Phi$ is the identity on $U'$, hence $(\Phi^{-1})_*\circ \Phi_* : F^M(X)\to F^M(X)$ is the identity. Since $F^M(X)\to F^M(U')$ is injective, we can conclude that $\Phi_*$ is an isomorphism.
\end{proof}
Recall that, by definition, a {\em homotopy sheaf} is a strictly $\AAA^1$-invariant Nisnevich sheaf of abelian groups over the category of smooth $k$-schemes; we denote by $\HI(k)$ the category of such sheaves. There is a canonical functor
\begin{center}
$\sigma^{\infty}:\HI(k) \to \mathfrak{M}_k^{\operatorname{MW}}$
\end{center}
thanks to Theorem \ref{ThmDeg}. The previous theorem can be generalized as follows:

\begin{The}
	Let $X$ be a proper smooth integral scheme over $k$. Let $F\in \HI(k)$ be a homotopy sheaf, then $F(X)$ is a birational invariant of $X$.
\end{The}
\begin{proof}
The proof is the same as Theorem \ref{BirInv} thanks to \cite[Cor. 6.4.6]{Mor05}.
\end{proof}

  \bibliographystyle{alpha}
  \bibliography{exemple_biblio}

%\bibliographystyle{plain} % D'autres styles sont disponibles. Notez que les distributions LaTeX n'incluent généralement pas de styles de bibliographies francisés ; vous aurez donc des bibliographies en anglais.
%\bibliography{biblioo} % Remplacer "biblio" par le nom de votre fichier de références bibliographiques.

\end{document}